\documentclass[11pt]{article}

\usepackage[utf8]{inputenc} 
\usepackage[T1]{fontenc}    
\usepackage[colorlinks,linkcolor={black},citecolor={black},urlcolor={black}]{hyperref}
\usepackage{url}           
\usepackage{amsfonts}       
\usepackage{bm}       
\usepackage{xcolor}       
\usepackage{enumerate}

\usepackage[margin=2.5cm]{geometry} 
\usepackage{graphicx}
\usepackage{amsmath, amsthm, amssymb}
\usepackage{mathtools}
\usepackage{dsfont}
\usepackage{datetime}
\usepackage{subcaption}

\definecolor{jan}{rgb}{0.0,0.3,0.8}

\def\tand{\quad{\rm and}\quad}

\DeclarePairedDelimiter{\norma}{\lvert}{\rvert} 
\DeclarePairedDelimiter{\normad}{\lVert}{\rVert} 
\DeclarePairedDelimiter{\abs}{\lvert}{\rvert}
\DeclarePairedDelimiter{\tond}{(}{)} 
\DeclarePairedDelimiter{\quadr}{[}{]}
\DeclarePairedDelimiter{\graf}{\{}{\}} 
\DeclarePairedDelimiter{\scal}{\langle}{\rangle}

\DeclareMathOperator*{\argmin}{arg\,min}

\DeclareMathOperator*{\hess}{\mathrm{D}^2}
\DeclareMathOperator*{\law}{law}

\newcommand{\numberset}{\mathbb}
\newcommand{\N}{\numberset{N}}

\newcommand{\R}{\numberset{R}}

\renewcommand{\epsilon}{\varepsilon}

\newcommand{\dive}{\mathbf{\nabla} \cdot}
\newcommand{\dom}[1]{\mathrm{dom}(#1)}

\newcommand{\pp}{\mathcal{P}}

\newcommand{\expe}[1]{\mathbb{E}\bigl[#1\bigr]}

\newcommand{\kldiv}[2]{\mathcal{D}_{\mathrm{KL}}\tond*{#1\,\middle\|\,  #2}}
\newcommand{\fistwo}[2]{\mathcal{I}_2\tond*{#1\,\middle\|\,  #2}}
\newcommand{\fisinf}[2]{\mathcal{I}_\infty\tond*{#1\,\middle\|\,  #2}}

\newcommand{\fisp}[2]{\mathcal{I}_p\tond*{#1\,\middle\|\,  #2}}

\DeclareMathOperator{\supp}{supp}

\newcommand{\indic}{\mathds{1}}

\newcommand{\dd}{{\, \mathrm{d}}}
\newcommand{\ddd}{{\mathrm{d}}}
\newcommand{\ddt}{\frac{\mathrm{d}}{\mathrm{d}t}}

\newcommand{\ip}[1]{\langle {#1}\rangle}   
\newcommand{\bip}[1]{\big\langle {#1}\big\rangle}   

\numberwithin{equation}{section}
\theoremstyle{plain}
\newtheorem{theorem}{Theorem}[section]
\newtheorem{proposition}[theorem]{Proposition}
\newtheorem{lemma}[theorem]{Lemma}
\newtheorem{corollary}[theorem]{Corollary}

\newtheorem{definition}[theorem]{Definition}
\newtheorem*{defs*}{Definition}

\newtheorem*{remark}{Remark}
\newtheorem*{ass*}{Assumption}

\newcommand{\sP}{\mathsf{P}}

\newcommand{\bfF}{\mathbf{F}}

\newcommand{\bfV}{\mathbf{V}}

\newcommand{\cC}{\mathcal{C}}

\newcommand{\cF}{\mathcal{F}}

\newcommand{\cR}{\mathcal{R}}
\newcommand{\cS}{\mathcal{S}}
\newcommand{\cT}{\mathcal{T}}

\newcommand{\hcT}{\widehat{\mathcal{T}}}

\def\BS{\boldsymbol}

\def\bflambda{{\BS\lambda}}

\newcommand{\E}{{\mathbb E}}
\renewcommand{\P}{{\mathbb P}}

\newcommand{\qbr}{\overline{B_R}\times \overline{B_R} }
\newcommand{\clobal}{\overline{B_R}}

\DeclareMathOperator*{\esssup}{ess\,sup}

\title{$L^\infty$-optimal transport of anisotropic log-concave measures and exponential convergence in Fisher's infinitesimal model}

\author{Ksenia A.~Khudiakova
	\qquad Jan Maas 
	\qquad Francesco Pedrotti \thanks{Institute of Science and Technology Austria (ISTA),
		Am Campus 1, 3400 Klosterneuburg, Austria \quad \\ (
		\texttt{kseniia.khudiakova@ist.ac.at}	
		\texttt{jan.maas@ist.ac.at}
		\texttt{francesco.pedrotti@stat.math.ethz.ch}
		)}}

\begin{document}

\maketitle
\begin{abstract}
	We prove upper bounds on the $L^\infty$-Wasserstein distance from optimal transport between strongly log-concave probability densities and log-Lipschitz perturbations. 
In the simplest setting, such a bound amounts to a transport-information inequality involving the $L^\infty$-Wasserstein metric and the relative $L^\infty$-Fisher information. 
We show that this inequality can be sharpened significantly in situations where the involved densities are anisotropic. 
Our proof is based on probabilistic techniques using Langevin dynamics. 
As an application of these results, we obtain sharp exponential rates of convergence in Fisher’s infinitesimal model from quantitative genetics, generalising recent results by Calvez, Poyato, and Santambrogio in dimension 1 to arbitrary dimensions.
	\end{abstract}
    
\tableofcontents
	
\section{Introduction}

	Upper bounds on transport distances to log-concave probability densities play a central role in the theory of optimal transport and in applications in high-dimensional geometry and probability.
	
	One fundamental example is \emph{Talagrand's inequality} \cite{Talagrand:1996}, which provides a remarkable upper bound for the $2$-Wasserstein distance to the standard Gaussian measure $\gamma$. For all probability measures $\nu$ having finite relative entropy 
	$
		\kldiv{\nu}{\gamma} 
		= 
		\int 
			\log \frac{\ddd \nu}{\ddd \gamma}(x)
		\dd \nu(x),
	$
	Talagrand's inequality asserts that
		$W_2(\nu, \gamma) \leq
		\sqrt{2 \kldiv{\nu}{\gamma} }$.
	More generally, Otto and Villani \cite{ott-vil-2000} showed that 
	\begin{align}
		\label{eq:Otto-Villani}
		W_2(\nu, \mu) 
		\leq
		\sqrt{\frac{2}{\kappa}\kldiv{\nu}{\mu} }
	\end{align}
	for all $\nu$, whenever $\mu$ satisfies a logarithmic Sobolev inequality with constant $\kappa > 0$.
	This includes in particular the class of all $\kappa$-log-concave densities.
	(A probability density $\mu$ is said to be 
		\emph{$\kappa$-log-concave} for some $\kappa \in \R$, 
		if $\mu = e^{-U}$ where 
			$U : \R^d \to \R \cup \{+ \infty\}$ 
		is \emph{$\kappa$-convex}; 
		i.e.,~$x \mapsto U(x) - \frac{\kappa}{2}|x|^2$ is convex.)
	The main reason for the great interest of this inequality is that it implies dimension-free Gaussian concentration for $\mu$.
	
	Another seminal result of a similar flavour is 
		\emph{Caffarelli's contraction theorem} \cite{Caffarelli:2000},
	which asserts that 
	any $1$-log-concave probability density $\mu$ can be obtained as the image  (or push-forward) of the standard Gaussian measure $\gamma$ under a $1$-Lipschitz map $T : \R^d \to  \R^d$.
	In fact, the optimal transport map for the $W_2$-distance (the so-called Brenier map) does the job.
	This theorem is a powerful tool to transfer functional inequalities from the Gaussian measure to the large class of $1$-log-concave measures.

	\subsection{ \texorpdfstring{$L^\infty$}{}-optimal transport of log-concave densities}

	This paper deals with yet another class of bounds 
	on the transport distance to a log-concave reference density, 
	involving the transport distance $W_\infty$
	instead of the more common distance $W_2$.
	For probability measures
	$\mu, \nu$ on $\R^{d}$,
	$W_{\infty}(\mu, \nu)$ can be defined in probabilistic terms by
	\begin{align*}
		W_{\infty}(\mu, \nu)
		= 
		\inf_{X,Y} 
		\Big \{
		\esssup_{\omega \in \Omega} | X(\omega) - Y(\omega) |
		\Big\} \, ,
	\end{align*}
	where the infimum runs over 
	all $\R^{d}$-valued random vectors 
	$X$ and $Y$
	defined on the same probability space 
	$(\Omega, \cF, \P)$
	with
	$\law(X) = \mu$ and $\law(Y) = \nu$.

	Our goal is to obtain quantitative bounds on the transport distance $W_\infty(\mu,\nu)$
	to a log-concave reference density $\mu$ for a large class of measures. 
	The following prototypical example is the simplest special case of our main result; see Corollary \ref{cor:log-lip-perturb-coupl} below.
	Under more restrictive assumptions, this bound was recently obtained in \cite{gan-tha-upa-2023},  \cite[Lem.~3.6] {aga-kal-kar-2023} and in \cite[Prop.~3.1]{Calvez-Poyato-Santambrogio:2023}.

\begin{proposition}
	\label{prop:simple}
	Let 
		$\mu$ and $\nu$ 
	be probability densities on $\R^d$.
	Suppose that  
		$\mu$ is $\kappa$-log-concave for some $\kappa > 0$,
	and that $\nu = e^{-H} \mu$,
	where $H \in C(\R^d)$ is
		$L$-Lipschitz for some $L < \infty$. 
	Then: 
	\begin{align}
		\label{eq:simple}
	W_{\infty} (\mu,\nu) \le \frac{L}{\kappa} \, .
	\end{align}
\end{proposition}

This bound is sharp, as can be seen by considering two shifted isotropic Gaussian measures.

Proposition \ref{prop:simple} can also be formulated as a 
functional inequality involving the 
\emph{$L^\infty$ relative Fisher information}
$\fisinf{\nu}{\mu}$ defined by
\begin{align*}
	\fisinf{\nu}{\mu} 
		= 
	\Bigl\| 
		\nabla \log \Bigl(\frac{\ddd \nu}{\ddd \mu} \Bigr)
	\Bigr\|_{L^\infty(\R^d, \mu)}
\end{align*}
for sufficiently regular densities $\nu \ll \mu$.
Indeed, Proposition \ref{prop:simple} asserts that
any probability density $\mu \in L^1_+(\R^d)$ that is 
$\kappa$-log-concave for some $\kappa > 0$
satisfies the 
	\emph{$L^\infty$ transport-information inequality} 
\begin{align*}
	W_\infty(\mu,\nu) \le \frac{1}{\kappa} \fisinf{\nu}{\mu} 
\end{align*}
for all sufficiently regular probability densities $\nu$. 
This inequality can be viewed as an $L^\infty$-analogue 
of well known $L^2$-based transport-information inequalities; see Section \ref{sec:transport-bounds} for more details.

\medskip

One of the main contributions of this paper is the insight that the estimate \eqref{eq:simple} can be improved significantly
when the involved probability densities are anisotropic.
Anisotropic densities are ubiquitous in applications, e.g., when densities are concentrated near a lower-dimensional manifold.
To formulate the improved estimate, it will be convenient to introduce some more notation.

Let $K \in \R^{d \times d}$ be a symmetric matrix.
A function $U : \R^d \to \R \cup \{+\infty\}$ is \emph{$K$-convex} if 
$x \mapsto U(x) - \frac12 \ip{x, Kx}$ is convex.
If $U \in C^2(\R^d)$, 
then $U$ is $K$-convex 
if and only if
	$\hess U(x) \succcurlyeq K$
for all $x \in \R^d$.
A function $\mu \in L_+^1(\R^d) \setminus \{0\}$ is said to be \emph{$K$-log-concave} if $\mu = e^{-U}$ for some $K$-convex function $U$.
The special case $K = \kappa I_d$ corresponds to the notions of $\kappa$-convexity and $\kappa$-log-concavity introduced above.
If $K = 0$, we recover the usual notions of convexity and log-concavity.

\smallskip

Let $A$ and $B$ be orthogonal subspaces satisfying $A \oplus B = \R^n$, and let $\sP_A$ and $\sP_B$ denote the corresponding orthogonal projections. 
The following result (see Corollary \ref{cor:directional} below) is a generalisation of Proposition \ref{prop:simple}, capturing different behaviour of the involved measures on the subspaces $A$ and $B$.
In the special case where $A = \R^n$ and $B = \emptyset$ we recover \eqref{eq:simple}.

\begin{theorem}
\label{thm:main-directional}
Let 
		$\mu$ and $\nu$ 
	be probability densities on $\R^d$.
Suppose that  
	$\mu$ is $K$-log-concave 
where 
	$K = \kappa_A \sP_A + \kappa_B \sP_B$ 
	for some $\kappa_A, \kappa_B > 0$,
and that $\nu = e^{-H} \mu$,
	with $H \in C(\R^d)$ satisfying,
	for some $L_A < \infty$,
	\begin{align*}
		\norma*{H(x)-H(y)}
			\leq 
		L_A |\sP_A (x-y) |
		\quad \text{for all }
		x, y \in \R^d \, .
	\end{align*}
Then:
\begin{equation*}
	W_{\infty}(\mu, \nu) \leq 
	\begin{cases}
		\frac{L_A}{\kappa_A} 
		&
		\text{ if } 
		\kappa_A \leq 2 \kappa_B \, ,
		\\
		\frac{L_A}
		{2 \sqrt{\kappa_B(\kappa_A - \kappa_B)}} 
		&
		\text{ if } 
		\kappa_A \geq 2 \kappa_B\, .
	\end{cases}
\end{equation*}
\end{theorem}

In the regime $1 \leq \frac{\kappa_A}{\kappa_B} \leq 2$, 
observe that the constants in the denominator 
depends only on the directional log-concavity constant $\kappa_A$, and not on the uniform log-concavity constant $\kappa_B$. 

Proposition \ref{prop:simple} and 
Theorem \ref{thm:main-directional}
will be proved as corollaries to a general criterion
	(Theorem \ref{thm:main-abstract-coupl}).
The proof is based on a probabilistic argument using careful estimates for Langevin dynamics for $\mu$ and $\nu$. 

While our main results are general, our investigation is partly motivated by applications to the long-term behaviour of Fisher's infinitesimal model from quantitative genetics, as will be discussed in Section \ref{sec:Fisher}. 
The improvement of Theorem \ref{thm:main-directional} over Proposition \ref{prop:simple} is crucial to obtain sharp rates of convergence in this model, as we will discuss below.

\subsection{Application to Fisher's infinitesimal model}
\label{sec:Fisher}

Fisher's infinitesimal model from quantitative genetics describes the distribution $F_n \in L^1_+(\R^d)$ of a $d$-dimensional trait $x \in \R^d$ in an evolving population at discrete times $n \in \N_0$.	
The trait distribution evolves according to the rule $F_{n+1} = \cT[F_n]$, where 	$\cT = \cS \circ \cR$
consists of a reproduction operator $\cR$ and a selection operator $\cS$ acting on $L^1_+(\R^d)$.
The reproduction operator $\cR$ is Fisher's infinitesimal operator given by
\begin{align*}
	\cR[F](x) 
	= 
	\int_{\R^{d} \times \R^{d}} 
	G\Big(x - \frac{x_1 + x_2}{2}\Big) 
	\frac{F(x_1) F(x_2)}{\|F\|_{L^1}}
	\dd x_1 \dd x_2
\end{align*} 
for $F \in L^1_+(\R^d)$ and $x \in \R^d$,
where $G(x) = (2\pi)^{-d/2} \exp(-|x|^2/2)$ is the standard Gaussian kernel on $\R^d$. We use the natural convention that $\cR[0] = 0$.
This operator describes sexual reproduction in a mean-field model where individuals mate independently and produce offspring whose traits are (isotropic) Gaussian centred at the average traits of their two parents.
The operator $\cR$ preserves the size of the population: $\| \cR[F]\|_{L^1} = \| F\|_{L^1}$ for all $F \in L_+^1(\R^d)$.
Selection effects are modelled using the multiplication operator $\cS$, which is given by
\begin{align*}
	\cS[F](x) 
	= e^{-m(x)} F(x)
\end{align*}
for a fixed mortality function $m : \R^d \to [0,\infty)$. This operator reflects the idea that individuals with certain traits have a higher survival probability than others. 
In this paper, $m$ will be strictly convex, which means that individuals with intermediate trait values have a higher survival probability. This is the regime of 
\emph{stabilising selection}.

Fisher's infinitesimal model was introduced in \cite{Fisher:1919} and explicitly formulated in \cite{Bulmer:1985}.
Though the model has been influential in quantitative genetics since it was proposed, 
it was proved only recently that the model emerges as a limit of models subject to the laws of Mendelian inheritance when the number of discrete loci tends to infinity \cite{Barton-Etheridge-Veber:2017}. 
We refer to \cite[Ch.~24]{Walsh:2018} for the biological background of various different infinitesimal models.

	\subsubsection*{Long-term behaviour}
	Significant recent progress has been obtained in understanding the long-term behaviour of the model as $n \to \infty$ under suitable assumptions on the mortality function $m$. 
	In particular, it is natural to ask whether there exists a (unique) probability distribution $\bfF$ that is \emph{quasi-invariant} in the sense that 
	$
	\cT[\bfF] = \bflambda \bfF
	$
	for some $\bflambda > 0$.
	Then one may ask whether the renormalised densities
	$F_n / \bflambda^n$
	converge to $\bfF$
	for a general class of initial probability distributions $F_0$, 
	and to quantify the speed of convergence using suitable metrics or functionals.
	
	A comprehensive investigation has been carried out in the special case of quadratic selection, namely $m(x) = \frac{\alpha}{2} |x|^2$ for some $\alpha > 0$ \cite{Calvez-Lepoutre-Poyato:2021}.
	In this situation, the model preserves the class of Gaussian distributions
	and it is shown that there exists a unique quasi-equilibrium $\bfF$, which is an explicit Gaussian distribution. 
	Moreover, the authors prove exponential convergence to $\bfF$ (in the sense of relative entropy) for general initial data. 
	
	The remarkable recent paper \cite{Calvez-Poyato-Santambrogio:2023} treats more general uniformly convex selection in dimension $1$. 
	Namely, under the assumption that $m : \R \to [0,\infty]$ satisfies  $m'' \geq \alpha$ for some $\alpha > 0$, the authors show the existence of a (non-explicit) $\beta$-log-concave quasi-equilibrium $\bfF$, without establishing its uniqueness.
	The parameter $\beta > \max\{\frac12, \alpha\}$ depends on $\alpha$ in an explicit way.
	Moreover, \cite{Calvez-Poyato-Santambrogio:2023} 
	uncovers a remarkable central role played by the 
	$L^\infty$ relative Fisher information.
	The authors show that the one-step contractivity estimate
	\begin{align}
		\label{eq:Fisher-contraction}
		\fisinf{\cT[F]}{\bfF}
		\leq 
			\big(\tfrac12 + \beta\big)^{-1}
		\, \fisinf{F}{\bfF}
	\end{align}
	holds for all $F \in L^1_+(\R^d)$. 
	This inequality immediately yields the exponential convergence bound
	$
	\fisinf{F_n}{\bfF}
	\leq 
	(\frac12 + \beta)^{-n}
	\fisinf{F_0}{\bfF}
	$
	for all initial distributions $F_0$ with $\fisinf{F_0}{\bfF} < \infty$. 
	Observe that the latter condition is a strong assumption on the initial datum $F_0$; e.g.,~if $G$ and $G'$ are 1-dimensional Gaussian distributions with different variances, then $\fisinf{G}{G'} = \infty$.

	\subsubsection*{Proof of the one-step contractivity}
	
	Let us briefly discuss the strategy of the proof of \eqref{eq:Fisher-contraction} from \cite{Calvez-Poyato-Santambrogio:2023}. 
	After proving the existence of a $\beta$-log-concave quasi-equilibrium $\bfF$, the authors consider the renormalised densities 
	$u_n := F_n / \bflambda^n \bfF$,
	which satisfy the recursive equation
	\begin{align*}
		u_{n+1}(x)
		=
		\int_{\R^{d} \times \R^{d}} 
		\frac{u_n(x_1) u_n(x_2)}{\|u_n \bfF\|_{L^1(\R^d)}}
		P(x_1, x_2;x)
		\dd x_1 \dd x_2 \, ,
	\end{align*}
	where $P(x_1, x_2; x)$ denotes the weighted transition rates from parental traits $(x_1, x_2)$ to a child with trait $x$.  
	These rates are given by 
	\begin{align}
		\label{eq:P-def}
		P(x_1, x_2;x)
		= 
		\frac{1}{Z(x)}
		\bfF(x_1) 
		\bfF(x_2)
		G\Big(x - \frac{x_1 + x_2}{2}\Big)\, ,
	\end{align}
	where $Z(x) = \int_{\R^{d} \times \R^{d}} \bfF(x_1) 
	\bfF(x_2)
	G\big(x - \frac{x_1 + x_2}{2}\big) \dd x_1 \dd x_2$ denotes the normalising constant which ensures that $P(\cdot;x)$ is a probability distribution on $\R^d \times \R^d$ for all $x \in \R^d$.
	
	The proof of the one-step contractivity estimate \eqref{eq:Fisher-contraction} relies on two key inequalities.
	Firstly, for all strictly positive initial data 
		$u_0 \in C^1(\R^d)$ 
	and all $x, \tilde x \in \R^d$,
	it is shown in 
		\cite[Lem.~2.4]{Calvez-Poyato-Santambrogio:2023} 
	that
	\begin{align}
		\label{eq:key-1}
		| \log u_1(x) - \log u_1(\tilde x) |
		\leq
		\big\| \nabla \log u_0 \big\|_{L^\infty(\R^d)} \,
		W_{\infty,1}\big( P(\cdot;x) , P(\cdot; \tilde x)  \big) \, .
	\end{align}
	Here, $W_{\infty,1}$ denotes the $\infty$-Wasserstein metric 
	over the base space $\R^{2d}$ endowed with the norm 
		$|(x_1, x_2)|_1 := |x_1| + |x_2|$, 
	with $|x_i|$ denoting 
		the Euclidean norm of 
		$x_i \in \R^d$ for $i = 1,2$. 
	While \eqref{eq:key-1} is stated in  \cite{Calvez-Poyato-Santambrogio:2023}
	for $d =1$, the proof extends verbatim to arbitrary dimensions.
	
	The second key inequality from \cite{Calvez-Poyato-Santambrogio:2023} 
	is a sharp bound on the $W_{\infty,1}$-distance appearing in the above inequality. 
	Namely, in the special case $d = 1$, it is shown that, for all $x, \tilde x \in \R$,
	\begin{align}
		\label{eq:key-2}
		W_{\infty,1}\big( P(\cdot;x) , P(\cdot; \tilde x)  \big)
		\leq 
		\big(\tfrac12 + \beta\big)^{-1}
		 \, | x - \tilde x | \, .
	\end{align}
	The inequalities \eqref{eq:key-1} and \eqref{eq:key-2} combined yield the crucial one-step contractivity inequality \eqref{eq:Fisher-contraction} for the $L^\infty$ relative Fisher information.

	However, 
	as pointed out in \cite[Rem.~1.6]{Calvez-Poyato-Santambrogio:2023}, 
	there are non-trivial obstacles that prevent an extension of the proof of \eqref{eq:key-2} to higher dimensions.
	The reason is that this proof employs the Brenier map (the optimal transport map for the $W_2$-distance), which satisfies the Monge-Amp\`ere equation. The required $L^\infty$-bound on the Brenier map between $P(\cdot;x)$ and $P(\cdot; \tilde x) \in L^1_+(\R^{2})$ is then obtained by using a maximum principle for the Monge-Amp\`ere equation in convex but not uniformly convex domains, exploiting recent progress on the regularity theory for the Monge-Amp\`ere equation in two-dimensional domains with special symmetries \cite{Jhaveri:2019}.

	\subsubsection*{Results}
	
    In this paper we obtain a sharp multi-dimensional version of \eqref{eq:key-2} by a completely different (probabilistic) method. 
	In fact, the (backward-in-time) transition kernels from different points $x, \tilde x\in\R^d$ have an intrinsic anisotropic nature, they are strongly log-concave, and they can be expressed as log-Lipschitz perturbations of each other.
	Therefore, we can derive the desired sharp bound from Theorem \ref{thm:main-directional}.
	Using the notation from above, we first establish the existence of a quasi-invariant distribution in the multi-dimensional setting.
	
	\begin{theorem}[Existence of a quasi-equilibrium]
		\label{thm:main-existence}
		Let $m \in C^1(\R^d)$ be $\alpha$-convex
		for some $\alpha > 0$.
		Then there exist
		$\bflambda \in (0,1)$ 
		and
		a probability density $\bfF \in L_+^1(\R^d)$ 
		such that 
		$\cT [\bfF] = \bflambda \bfF$.
		Moreover, $\bfF$
		is $\beta$-log-concave, where 
			$\beta > \max \bigl\{ \frac12, \alpha \bigr\}$ 
		satisfies $\beta = \alpha + \frac{\beta}{\frac12 + \beta}$.
	\end{theorem}
	
	The proof of this result adapts the arguments from  \cite{Calvez-Poyato-Santambrogio:2023}, 
	where the corresponding result was obtained for $d = 1$.  
	The key technical tool is the $L^\infty$-transport bound from 
		Theorem \ref{thm:main-directional}, 
	which yields a Cauchy property for a sequence of iterates, 
	and hence a candidate quasi-equilibrium. 
	The properties of the $L^\infty$ relative Fisher information require us to work first with a localised problem on a bounded domain, and subsequently identify a quasi-equilibrium for the original operator $\cT$ 
	by an approximation procedure.
	The extension of this argument to higher dimensions brings additional technicalities to deal with the boundedness of the domains and to show tightness of a sequence of quasi-equilibria.

	As Theorem \ref{thm:main-existence} yields the existence 
	of a quasi-equilibrium $\bfF$,
	we can define the weighted transition kernels $P(\cdot;x)$ 
	by 
	$P(x_1, x_2;x)
		= 
		\frac{1}{Z(x)}
		\bfF(x_1) 
		\bfF(x_2)
		G\big(x - \frac{x_1 + x_2}{2}\big)$
	as in \eqref{eq:P-def}, 
	where $Z(x)$ denotes a normalising constant.
	Using Theorem \ref{thm:main-directional} we obtain the following $d$-dimensional generalisation of \eqref{eq:key-2}.
	
	\begin{theorem}[$W_\infty$-contractivity]
		\label{thm:main-W-infty}
		Let $m \in C^1(\R^d)$ be $\alpha$-convex 
		for some $\alpha > 0$.
		Then:
		\begin{align*}
			W_{\infty}
			\bigl( 
				P(\cdot; x) , P(\cdot; \tilde x)
				\bigr) 
				\le 
				2^{-1/2}
				\big(\tfrac12 + \beta\big)^{-1}
				| x - \tilde x|
				\,
		\end{align*}
		for all $x, \tilde x\in \R^d$,
		where $\beta > 
		\max\{\frac12,\alpha\}$ satisfies $\beta = \alpha + \frac{\beta}{\frac12 + \beta}$.
	\end{theorem}
         
	Since $W_{\infty,1} \leq \sqrt{2} W_\infty$ 
 in view of the trivial inequality 
 	$|(x_1,x_2)|_1 
		\le 
	\sqrt{2}\,\norma{(x_1,x_2)}$,
	this result implies the desired bound \eqref{eq:key-2}.
	Consequently, 
	the main conclusions of \cite{Calvez-Poyato-Santambrogio:2023} carry over to multi-dimensional traits.
	The following result summarises these conclusions.

	\begin{corollary}\label{cor:main-conv-fisher-model}
		Let $m \in C^1(\R^d)$ be $\alpha$-convex 
		for some $\alpha > 0$, and
		let $(\bflambda, \bfF)$ be as in Theorem \ref{thm:main-existence}.
		Take $0 \neq F_0 \in L^1_+(\R^d) $
        with
			$\fisinf{F_0}{\bfF} < \infty$, 
		and set
			$F_n = \cT^n[F_0]$
		for $k \geq 0$. 
		Then:
		\begin{enumerate}[(i)]
			\item \label{it:fis-inf-contr} ({Convergence of the relative $L^\infty$-Fisher information}) \
			For all $n \in \N$ we have
			\[
				\fisinf{F_n}{\bfF} 
					\leq
				\big(\tfrac12 + \beta\big)^{-n}
				\fisinf{F_0}{\bfF} \, . 
			\] 
			\item \label{it:kl-mass-dec} ({Convergence of the relative entropy}) \
			There exists a constant $C > 0$ depending on $F_0$  such that for all $n \in \N$ we have
			\begin{align*}
				\kldiv{ \frac{F_n}{\| F_n \|_{L^1}}}
					   {\bfF} 
					& \le 
				C \big(\tfrac12 + \beta\big)^{-2n}
				\tand
				\abs*{
						\frac{ \| F_n \|_{L^1} }
								{ \| F_{n - 1} \|_{L^1} }
						 - \bflambda} 
				\leq 
				C 
				\big(\tfrac12 + \beta\big)^{-n} \, . 
			\end{align*}
		\end{enumerate}
		
	\end{corollary}

One may wonder whether analogues of the contraction property  	
in \eqref{it:fis-inf-contr} hold with the same rate for functionals other than $\fisinf{\cdot}{\bfF}$, such as the relative entropy and the 
relative $L^2$-Fisher information.
In Section \ref{sec:different_metrics} we show that this is not the case, not even in the setting of quadratic selection 
	($m(x) = \frac{\alpha}{2}|x|^2$) and Gaussian initial data. 
We refer the reader to Section \ref{sec:different_metrics} for the details.

\subsection{Structure of the paper}

Section \ref{sec:transport-bounds} deals with 
$L^\infty$-optimal transport bounds for perturbations of log-concave densities, containing a general criterion (Theorem \ref{thm:main-abstract-coupl}) 
and the proofs of 
	Proposition \ref{prop:simple} and Theorem \ref{thm:main-directional}.
The applications to Fisher's infinitesimal model, and in particular the proof of Theorems \ref{thm:main-existence} and \ref{thm:main-W-infty} and Corollary \ref{cor:main-conv-fisher-model}, can be found in Section \ref{sec:application-fisher}.
The discussion after Corollary \ref{cor:main-conv-fisher-model} is expanded in  Section \ref{sec:different_metrics}, which deals with the relative $L^2$-Fisher information and the relative entropy instead of the relative $L^\infty$-Fisher information.
Finally, Section \ref{sec:peaks-log-conc} contains two lemmas on log-concave distributions that are used in the proof of Theorem \ref{thm:main-existence} in Section \ref{sec:application-fisher}.
	
	\subsection{Notation and preliminaries}
	\label{sec:notation}

	Let $L^1_+(\R^d)$ denote the cone of non-negative functions in $L^1(\R^d)$.
	Throughout the paper, we identify (probability) densities in $L_+^1(\R^d)$ with the corresponding (probability) measures.
	
	Weak convergence of densities (or measures) denotes convergence in duality with bounded continuous functions.
	We will frequently use that $(\mu,\nu) \mapsto W_\infty(\mu,\nu)$ 
	is jointly continuous with respect to weak convergence of probability measures. This follows from the corresponding result for $W_p$, since $W_p \to W_\infty$ pointwise as $p \to \infty$; see \cite{Givens-Shortt:1984}.

	\begin{definition}
		\label{def:infinity-Fisher}
    Suppose that $\mu \in L_+^1(\R^d)$ is a
	density,
	not necessarily normalised,
	such that $\supp \mu$ is closed and convex.
	If $\nu \in L_+^1(\R^d)$
	satisfies 
		$\nu \ll \mu$ 
	and 
		$ \log \bigl(\frac{\ddd\nu}{\ddd\mu}\bigr) = f$ 
	$\mu$-a.e.~for some Lipschitz function $f\colon \supp \mu \to \R$,
	then 
	\begin{equation}\label{eq:rel-Linf-fish-info}
		\fisinf{\nu}{\mu} 
			:= 
			\sup\biggl\{ 
				\frac{|f(x) - f(y)|}{|x - y|}
				\ : \ 
				x, y \in \supp \mu\,, \ x \neq y
			\biggr\} 
			\, .
	\end{equation}
	Otherwise, $\fisinf{\nu}{\mu} := + \infty$.
	\end{definition}	
	
	\begin{remark}
	In particular, if $\nu \ll \mu$ and $\log \bigl(\frac{\ddd\nu}{\ddd\mu}\bigr) = f$ $\mu$-a.e.~for some  $f \in C^1(\supp \mu)$, then 
		$\fisinf{\nu}{\mu} = \| \nabla f \|_{L^\infty(\R^d, \mu)}$.
	\end{remark}

    The relative entropy (or Kullback-Leibler divergence) of a probability density $\nu$ with respect to a probability density $\mu$ is defined by
    \begin{equation}\label{eq:rel-entr}
			\kldiv{\nu}{\mu} = 
			\begin{cases}
				\displaystyle 
    \int_{\R^d} 
				\rho \log \rho
		\dd \mu
     &\text{if } \nu \ll \mu \text{ 
			with $\rho:= \frac{\ddd\nu}{\ddd\mu}$} \, ,
				\\
				+ \infty &\text{otherwise}\, .
			\end{cases}
		\end{equation}

		$B_r(x)$ denotes the open ball of radius $r > 0$ around $x \in \R^d$. 
		Its closure will be denoted $\overline B_r(x)$

		$\gamma_{\mu,C}$ denotes the centred Gaussian density with mean $\mu \in \R^d$  and covariance matrix $C \in \R^{d \times d}$.
		If $\mu = 0$ we simply write $\gamma_{C}$.

	The following well-known property of log-concave densities will be useful in the sequel; see, e.g.,~\cite[Thm.~3.7.2]{sau-wel-2014}.

	\begin{lemma}[Preservation of log-concavity]
	\label{lem:convolution}
		For $i = 1, 2$, 
		let $\mu_i \in L_+^1(\R^d)$ be $K_i$-log-concave for some matrix 
		$K_i \in \R^{d\times d}$ with $K_i \succ 0$. 
		Then $\mu_1 * \mu_2$ is $K$-log-concave with
		\begin{align*}
			K^{-1} = K_1^{-1} + K_2^{-1} \, .
		\end{align*}
	\end{lemma}

	We also use the following well-known result in the reverse direction; see \cite[Lem.~1.3]{Eldan-Lee:2018}.

	\begin{lemma}[Log-convexity along the heat flow]\label{lem:upp-bound-hess-pot-gauss-conv}
		Let $\mu$ be a probability measure on $\R^d$. 
		For any $t > 0$ the probability density 
			$\mu_t := \mu * \gamma_{t I_d}$ 
		is $(-\frac1t)$-log-convex, 
		in the sense that, for all $x \in \R^d$,
		\begin{equation}
			\hess \bigl(
					-\log \mu_t(x) \bigr) 
				\preccurlyeq   
				\frac{1}{t}I_d \, .
		\end{equation}
	\end{lemma}

	\section{\texorpdfstring{$L^\infty$}{}-optimal transport of log-concave measures}	
	\label{sec:transport-bounds}

	In this section we present several bounds for the $\infty$-Wasserstein distance $W_\infty(\mu,\nu)$ 
	between 
	a log-concave measure $\mu$ and a log-Lipschitz perturbation $\nu$.
	Unless specified otherwise, 
	the Wasserstein distance is taken with respect to the Euclidean distance on the underlying space.
	Our bounds will be derived from the following general criterion.

	\begin{theorem}
	\label{thm:main-abstract-coupl}
Let $\mu$ and $\nu$ be probability densities on $\R^d$ 
satisfying the following assumptions:
\begin{enumerate}[(i)]
	\item $\mu$ is $K$-log-concave for some matrix 
	$K \in \R^{d\times d}$ with $K \succ 0$. 
	\item \label{it:directional-Lip} $\nu = e^{-H} \mu$ with 
	$H \in C(\R^d)$ satisfying
	\begin{align*}
		\abs*{H(y)-H(x)} \le \ell(x-y)
		\quad \text{for all }
		x, y \in \R^d \, ,
	\end{align*}
	for some
		positively 1-homogeneous 
		function
		$\ell \in C(\R^d)$.
\end{enumerate}
Then we have 
\begin{equation*}
	W_{\infty}(\mu, \nu) \leq M \, ,
\end{equation*}
where 
\[
M \coloneqq \sup_{z\in \R^d} 
	\Bigl\{ 
	| z | 
	\, : \,
		\ip{z, Kz}  \leq \ell(z) 
		\Bigr\} \, .
\]
\end{theorem}

    \begin{remark}
			Note that the assumptions imply that $H$ is Lipschitz continuous with Lipschitz constant 
				$L := \sup_{|z| = 1} \ell(z)$.
            A possible choice of $\ell$ is given by $\ell(z) = L \norma{z}$.
			However, it is important to allow for other choices of $\ell$ which take anisotropy into account. 
			This will indeed be crucial to get optimal bounds in our application to the Fisher model.
            When $H\in C^1(\R^d)$, 
			the assumed bound on $H$ can be written equivalently as 
            \begin{align*}
				\scal{\nabla H(x), z} \le \ell(z)
				\quad \text{for all }
				x, z \in \R^d\, .
			\end{align*}
        \end{remark}
    \begin{proof}
		The proof consists of three steps.
		\smallskip

    \textit{Step 1.} \
    Suppose first that $\mu = e^{-U}$ for some 
		$U \in C^2(\R^d)$   
	such that $\nabla U$ is Lipschitz,
	and that $H\in C^1(\R^d)$.
	It then follows from the standard theory of stochastic differential equations 
		\cite[Thm.~5.2.9]{Karatzas-Shreve:1991} 
	that 
    there exists a unique strong solution to the following system of SDEs, driven by \emph{the same Brownian motion} $B_t$, for all times $t \geq 0$:
		\begin{align}
			\label{eq:first-lang-log-conc}
			\dd X_t 
			& = 
			-\nabla U(X_t) \dd t 
			+
			\sqrt{2} \dd B_t \, , \quad 
			& X_0 & \sim \nu\, ,
			\\
			\label{eq:second-lang-log-lip-pert}
			\dd Y_t 
			&	= 
			-\nabla U(Y_t) \dd t 
			- \nabla H(Y_t) \dd t 
			+
			\sqrt{2} \dd B_t\, , 
			\quad 
			& Y_0 & = X_0\,.
		\end{align}
		Subtracting these equations in their integral form
		we note that
		the Brownian term vanishes,
		and since $X$ and $Y$ have a.s.~continuous sample paths,
		we infer that
		the sample paths of
		$Z := X - Y$ 
		are continuously differentiable a.s.
		Using the chain rule and our assumptions, we find
		\begin{align*}
			\frac12 \ddt  |Z_t|^2 
			& =  - \bip{X_t-Y_t,\nabla U (X_t) - \nabla U(Y_t) }
			+ \bip{X_t-Y_t, \nabla H(Y_t)}
			\\
			& \leq 
			- \scal{Z_t, K \,Z_t} + \ell(Z_t) \, .
		\end{align*}
		
		Observe now that, for any differentiable function 		
			$h\colon \R_{\geq 0}\to \R_{\geq 0}$ 
		with $h(0)=0$ 
		we have
		$
		\sup h = \sup_{x \geq 0}\{ h(x) : h'(x)\geq 0 \}.
		$
		Applying this identity to 
		$h(t) = \frac12|Z_t|^2$, we obtain 
		\begin{align}
		\label{eq:xy-bound}
			|X_t-Y_t| \le M  
			\quad\text{for all } t\ge 0\, .
		\end{align}
		
Since $\mu$ is strongly log-concave and $X_0$ has finite second moment, 
$\law(X_t)$ converges to $\mu$ in $W_2$-distance as $t \to \infty$, 
hence weakly. 
Using the joint lower semicontinuity of $W_\infty$ with respect to weak convergence \cite{Givens-Shortt:1984}
we deduce that $W_\infty(\mu, \nu) \leq M$.

\smallskip
\textit{Step 2.} \
We now remove the extra assumptions on $\mu$. 
To this end, set 
$\mu_n = \mu * \gamma_{\frac{1}{n}I_d}$ 
and define the probability density 
	$\nu_n 
		\propto e^{-H}\mu_n$.
Note that $U_n = - \log \mu_n$ is smooth
with
\[
	K_n \coloneqq \tond*{K^{-1} + \frac{1}{n}I_d}^{-1} \preccurlyeq \hess U_n \preccurlyeq nI_d			
\]
by Lemma \ref{lem:convolution} and \ref{lem:upp-bound-hess-pot-gauss-conv}.
Therefore, we are in a position to apply Step 1 and we obtain the bound 
	$W_\infty(\mu_n, \nu_n) \le M_n$,
where
\[
M_n \coloneqq \sup_{z\in \R^d} 
	\Bigl\{ 
	| z | 
	\, : \,
		\ip{z, K_nz} 
		\leq \ell(z)  
		\Bigr\} \, .
\]
        
Note that $\mu_n \to \mu$ weakly. 
Moreover, Lemma \ref{lem:strong-weak-conv-log-conc-gauss-conv}
below implies that 
$\nu_n \to \nu$ weakly too. 
Hence, using again the joint lower semicontinuity of $W_\infty$ with respect to weak convergence 
we find
\begin{align*}
	W_\infty(\mu, \nu)
			\leq 
	\liminf_{n \to \infty}
		W_\infty(\mu_n, \nu_n)
			\leq 
	\liminf_{n \to \infty}
		M_n \, .
\end{align*}
It thus remains to show that $M_n \to M$.

For this purpose, we define the sets
\begin{align*}
	C_n = 
		\graf*{z \in \R^d \, : \, \ip{z, K_n z}  \leq \ell(z) }
	\tand
	C = 
		\graf*{z \in \R^d \, : \,  \ip{z, K z}  \leq \ell(z) } 
	\, .
\end{align*}
Since $t \mapsto t^{-1}$ is operator monotone (see, e.g., \cite[Lem.~2.7]{Carlen:2010}), 
we have
$
	\ip{z, K_n z} 
		\geq 
	\ip{z, K_{n-1} z}
$
for all $z$,
hence
	$C_n \subseteq C_{n-1}$
and $M_n \leq M_{n-1}$.
Moreover, since $\ip{z, K_n z}  \to \ip{z, K z}$ 
monotonically for all $z$,
we have 
	$C = \bigcap_n C_n$.

Using the continuity and the positive $1$-homogeneity 
of $\ell$, 
we infer that the sets $C_n$ are non-empty and compact.
Consequently, there exists
		$z_n \in C_n \subseteq C_1$ with
		$|z_n| = M_n$.
Since $C_1$ is compact, we may extract a subsequence
$\{z_{n_k}\}_k$ converging to some $\hat z \in C_1$.
Since each $C_m$ is closed, and since $z_{n_k} \in C_m$
whenever ${n_k} \geq m$, it follows that $\hat z \in C_m$, hence $\hat z \in \bigcap_m C_m = C$.
Therefore, $M \geq |\hat z| 
	= \lim_{k \to \infty} |z_{n_k}| 
	= \lim_{k \to \infty} M_{n_k}$.
Since $M \leq M_n \leq M_{n-1}$ for all $n$,
it follows that $\lim_{n \to \infty} M_{n} = M$.

\smallskip
\textit{Step 3.} \
We remove the differentiability assumptions on $H$. 
Write
\[
L := \sup_{\norma{x}=1} \ell(x) < \infty\, ,
\]
so that $H$ is $L$-Lipschitz.
Let $j \colon\R^d\to \R_{\ge 0}$ be a smooth mollifier supported in the unit ball of $\R^d$.
We write 
$j_n (x) := n^d j(nx)$ and $H_n := j_n * H$,
so that
\[
H_n(x) 
	= 
	n^d \int_{\R^d} 
			H(x-y) 
			j(ny) 
		\dd y 
	= 
		\int_{\R^d} 
			H\tond*{x-\frac{y}{n}} 
			j(y) 
		\dd y \, . 
\]
Since
$\supp j \subseteq B_1(0)$,
we have for  all $x, y \in \R^d$,
\begin{align}\label{eq:unif-conv-Hn}
	\abs*{H_n(x) - H(x)} 
			& \le
		\frac{L}{n}\, ,
	\\ \label{eq:lip-Hn}
	\abs*{H_n(x) - H_n(y)} 
			& \le 
		\ell(x-y) \leq {L} \norma*{x-y}\, .
\end{align}
	Define the probability measures $\nu_n \propto e^{-H_n}\mu$. 
	Since $H_n$ is a smooth function satisfying \eqref{eq:lip-Hn}, an application of Step 2 yields
	\[
		W_\infty(\mu, \nu_n) \leq M \, .
	\]
	Hence, since $W_\infty$ is jointly weakly lower semicontinuous,
	it suffices to show that $\nu_n \to \nu$ weakly. 
	For this purpose, it is in turn sufficient to prove that  
	$e^{-H_n}$ converges to $e^{-H}$ in $L^1(\mu)$, which we will do next.

Fix $\epsilon > 0$.
Since $\mu$ is $\kappa$-log-concave with $\kappa > 0$,
we have $- \log \mu(x) \geq \frac{\kappa}{2}| x - \bar x|^2$ for some $\bar x \in \R^d$.
Furthermore, since $\abs*{H(x)} \leq |H(0)| + L |x|$, \eqref{eq:unif-conv-Hn} implies that 
	$
	\abs*{H_n(x)}
		\le C 
				+ L \norma*{x} 
	$
with $C := \abs*{H(0)} + L$.
Therefore, there exists $R > 0$ such that, for all $n \geq 1$, 
\begin{align*}
	\int_{B_R(0)^c} e^{-H_n} \dd \mu 
	+ 
	\int_{B_R(0)^c} e^{-H} \dd \mu 
	\leq \frac{\epsilon}{2} \, .
\end{align*}
Furthermore, since the function $x \mapsto e^{-x}$ is uniformly continuous on bounded intervals,
\eqref{eq:unif-conv-Hn} implies that
there exists $\bar n \geq 1$ such that for all $n\ge \bar n$,
\[
	\sup_{x\in \overline{B}_R(0)}
	\bigl|e^{-H_n(x)}-e^{-H(x)}\bigr|
	\le 
	\frac{\epsilon}{2}\, .
\]
Consequently, for $n \geq \bar n$,
\begin{align*}
			\int_{\R^d} 
				\bigl| 
					e^{-H_n} - e^{-H}
				\bigr|  \dd \mu
			\le 
			\int_{B_R(0)^c}
				e^{-H_n} \dd \mu
			+
			\int_{B_R(0)^c}
				e^{-H} 
				\dd \mu
			+ 
			\sup_{\overline{B}_R(0)}
				\norma*{e^{-H_n}-e^{-H}}
			\leq \epsilon
			\, ,
\end{align*}
which implies that $e^{-H_n} \to e^{-H}$ in $L^1(\mu)$ as $n \to \infty$.
\end{proof}

\begin{lemma} \label{lem:strong-weak-conv-log-conc-gauss-conv}
	Let $\mu \in L_+^1(\R^d)$ be a $K$-log-concave probability density 
	for some matrix 
		$K \in \R^{d\times d}$ with $K \succ 0$,
	and define
		$\mu_n = \mu* \gamma_{\frac{1}{n}I_d}$
	for $n \geq 1$. 
	Then
	\[
		\int f\dd\mu_n \to \int f\dd\mu
	\]
	for all continuous functions $f\colon \R^d \to \R$ satisfying
	$\abs*{f(x)}\le C\exp\tond*{C\norma{x}}$
	for some $C>0$.
\end{lemma}

\begin{proof}
	We show first that the integrals above are finite.
	Let $f$ be as in the statement, 
	let 
		$X\sim \mu$ and $Z\sim \gamma_{I_d}$ be independent, 
	and set 
		$X_n = X+ \frac{Z}{\sqrt{n}}$. 

	Let $\kappa > 0$ be the smallest eigenvalue of $K$, 
	and fix $\kappa' \in (0,\kappa)$.
	It follows from Lemma \ref{lem:convolution} 		
	that 
		$\mu_n$ is $\kappa'$-log-concave 
		for all $n$ sufficiently large.
	Therefore, the Bakry-\'Emery criterion implies that
		the measures $\mu$ and $\mu_n$ satisfy a logarithmic 
		Sobolev inequality with the same constant.
	Using this, the growth assumption on $f$, 
	and the fact that 
		$\expe{X_n} = \expe{X}$,
	the so-called Herbst argument 
		\cite[Prop.~5.4.1]{Bakry-Gentil-Ledoux:2014}
	implies that 
		$f(X_n) \in L^2$	
	and that the sequence $\{f(X_n)\}_n$ is bounded in $L^2$. 
	In particular, $f(X), f(X_n) \in L^1$, 
	hence the integrals above are finite.
	It remains to show that
	\[
		\expe{f(X_n)} \to \expe{f(X)}.
	\]
	For this purpose, note first that 
	$X_n \to X$ in probability.
	Since $f$ is continuous, 
	$f(X_n) \to f(X)$ in probability as well; 
		see, e.g.~\cite[Lem.~5.3]{Kallenberg:2021}.
	Therefore, to conclude that 
		$f(X_n)\to f(X)$ in $L^1$
	it suffices to show that 
		$\{ f(X_n) \}_n$ is uniformly integrable;
	see, e.g.~\cite[Thm.~5.12]{Kallenberg:2021}.
	But this follows from the fact
	that the sequence $\{ f(X_n) \}_n$ is bounded in $L^2$,
	which we proved above.
\end{proof}

\subsection{Isotropic case}

The simplest non-trivial case of Theorem \ref{thm:main-abstract-coupl} is the following estimate, which we stated as Proposition \ref{prop:simple} above.

\begin{corollary}\label{cor:log-lip-perturb-coupl}
	Let $\mu$ and $\nu$ 
	be probability densities on $\R^d$.
	Suppose that  
		$\mu$ is $\kappa$-log-concave for some $\kappa > 0$,
	and that $\nu = e^{-H} \mu$,
	where $H \in C(\R^d)$ is
		$L$-Lipschitz for some $L < \infty$. 
	Then: 
	\begin{align}
		\label{eq:transport-bound}
		W_{\infty} (\mu,\nu) \le \frac{L}{\kappa}\, .
	\end{align}
\end{corollary}
	
	\begin{proof}            
    This is an application of Theorem \ref{thm:main-abstract-coupl} with $K = \kappa I_d$ and $\ell(z) = L |z|$. 
	\end{proof}

	The following result is a reformulation of Corollary \ref{cor:log-lip-perturb-coupl} as a functional inequality.
	
	\begin{theorem}[$\infty$-Transport-Information Inequality]
		\label{thm:funct-ineq}
		Let $\mu \in L_+^1(\R^d)$ be a $\kappa$-log-concave probability density for some $\kappa > 0$.
		Then the \emph{transport-information inequality}
		\begin{equation}\label{eq:linf-transp-inf-ineq}
			W_\infty(\mu,\nu) 
				\le 
			\frac{1}{\kappa} 
				\fisinf{\nu}{\mu}
		\end{equation}
		holds for all probability densities $\nu \in L_+^1(\R^d)$.
	\end{theorem}

	\begin{proof}
		Suppose that $\fisinf{\nu}{\mu} < +\infty$; otherwise there is nothing to prove.
		In view of Definition \ref{def:infinity-Fisher} there exists a Lipschitz function $h : \supp \mu \to \R$ with Lipschitz constant 
			$L := \fisinf{\nu}{\mu}$, 
			that agrees with 
			$\log\bigl( \frac{\ddd \nu}{\ddd \mu} \bigr)$ $\mu$-a.e.. 
		By the Kirszbraun theorem,
			$h$ can be extended to a Lipschitz function $H$ on $\R^d$ 
		with the same Lipschitz constant $L$.
		Since $\nu = e^{-H}\mu$, the result follows from
		Corollary \ref{cor:log-lip-perturb-coupl}.
	\end{proof}

    \begin{remark}
\label{rem:transport-information}
	The inequality \eqref{eq:linf-transp-inf-ineq} is an $L^\infty$-analogue of the well-known  
	$L^2$-based transport-information inequality
	\begin{align}
		\label{eq:WI2}
		W_2(\mu,\nu) 
			\le 
		\frac{1}{\kappa}
			\sqrt{\fistwo{\nu}{\mu}} \ ,
	\end{align}
	where $\fistwo{\nu}{\mu} 
		:= 
			\bigl\| 
				\nabla \log \bigl(\tfrac{\ddd\nu}{\ddd\mu} \bigr)
			\bigr\|_{L^2(\nu)}^2$ 
    denotes the $L^2$-relative Fisher Information 
	for sufficiently regular densities $\nu$.

	The latter inequality holds under the assumption that 
		$\mu$ satisfies a logarithmic Sobolev inequality
	$
		\kldiv{\nu}{\mu} \leq \frac{1}{2 \kappa} 
			\fistwo{\nu}{\mu},
	$
	and thus for every
	$\kappa$-log concave measure $\mu$
	by the Bakry--\'Emery criterion.
	To prove \eqref{eq:WI2}, note that 
	the logarithmic Sobolev inequality implies the
	transport-entropy inequality 
	$
		W_2(\mu, \nu) 
			\leq 
			\sqrt{
					\frac{2}{\kappa} 
					\kldiv{\nu}{\mu}
			} 
	$
	by the work of Otto and Villani \cite{ott-vil-2000}. 
	Combining these two inequalities immediately yields \eqref{eq:WI2}. 	
    
    In fact, by modifying the proof of Theorem \ref{thm:main-abstract-coupl}, in Appendix \ref{app-p-transp-inf}  we show that an analogue of \eqref{eq:WI2} holds for any $p\in [1,\infty]$.
	For a systematic study of transport-information inequalities 
	we refer to \cite{Guillin-Leonard-Wu:2009}.
\end{remark}

	\subsection{Anisotropic case}

We will now develop a more refined criterion, 
that yields improved bounds in situations where the measures behave differently in different directions.
Let $A$ and $B$ be non-empty subspaces of $\R^d$ that are orthogonal and satisfy $A \oplus B = \R^d$. 
Let $\sP_A$ and $\sP_B$ be the corresponding orthogonal projections.

\begin{corollary}
\label{cor:directional}
Let 
$\mu$ and 
$\nu$ be 
probability densities on $\R^d$ satisfying the following assumptions:
\begin{enumerate}[(i)]
	\item $\mu$ is $K$-log-concave, 
		with $K = \kappa_A \sP_A + \kappa_B \sP_B$ for some 
		$\kappa_A, \kappa_B > 0$.
	\item $\nu = e^{-H} \mu$ 
  		with $H \in C(\R^d)$ satisfying,
		for some $L_A < \infty$,
		\begin{align*}
			\norma*{H(x)-H(y)}
				\leq 
			L_A |\sP_A (x-y) |
			\quad \text{for all }
			x, y \in \R^d \, .
		\end{align*}
\end{enumerate}
Then:
\begin{equation}
	\begin{aligned}
	\label{eq:two-cases}
	W_{\infty}(\mu, \nu) 
	\leq 
	\begin{cases}
		\frac{L_A}{\kappa_A} 
		&
		\text{ if } 
		\kappa_A \leq 2 \kappa_B \, ,
		\\
		\frac{L_A}
		{2 \sqrt{\kappa_B(\kappa_A - \kappa_B)}} 
		&
		\text{ if } 
		\kappa_A \geq 2 \kappa_B\, .
	\end{cases}
	\end{aligned}
\end{equation}
\end{corollary}
	
	\begin{proof}
		Applying 
		Theorem \ref{thm:main-abstract-coupl}
		with 
		$K = \kappa_A \sP_A + \kappa_B \sP_B$
		and 
		$\ell(z) =  L_A |\sP_A z |$,
		we infer that
			$W_{\infty}(\mu, \nu) \leq M$,
		where 
		\begin{equation}
			M 
			:= 
			\sup_{ z_A , z_B \geq 0 }
			\Bigl\{
			\sqrt{z_A^2 + z_B^2 }
			\ : \
			\kappa_A z_A^2 + \kappa_B z_B^2  
			\leq L_A z_A 
			\Bigr\}\, .
		\end{equation}
		Performing the maximisation over $z_B$ first, 
		we observe that
		\begin{align*}
			M^2
			&= 
			\sup
			\Bigl\{
			z_A^2 
			+ 
			\frac1{\kappa_B}
			\bigl(L_A z_A - \kappa_A z_A^2\bigr) 
			\ : \
			0 \leq z_A \leq \frac{L_A}{\kappa_A }
			\Bigr\}
			\\& 
			=
			\frac{1}{\kappa_B}
			\sup 
			\Bigl\{
			p(z_A)
			\ : \
			0 \leq z_A \leq \frac{L_A}{\kappa_A }
			\Bigr\}
			\, ,
		\end{align*}
		where $
		p(z) = 
		L_A z
		- (\kappa_A - \kappa_B) z^2 
		$.
		We now distinguish two cases.
		
		If
		$\kappa_A \leq 2 \kappa_B$,
		then 
		$p$ is non-decreasing on the interval 
		$[0,L_A/\kappa_A ]$.
		Therefore, the supremum of $p$ on  
		$[0,L_A/\kappa_A ]$ is attained at the right endpoint of this interval, hence
		\begin{align*}
			M^2 = 
			\frac{1}{\kappa_B}
			p\Bigl(\frac{L_A}{\kappa_A}\Bigr)
			= \frac{L_A^2}{\kappa_A^2} \, .
		\end{align*}
		If
		$\kappa_A > 2 \kappa_B$,
		then 
		$p$ attains its global maximum 
		in the open interval
		$(0,L_A/\kappa_A )$,
		at $z := \frac{L_A}{2(\kappa_A - \kappa_B)}$.
		Therefore,
		\begin{align*}
			M^2 = 
			\frac{1}{\kappa_B}
			p\Bigl(\frac{L_A}{2(\kappa_A - \kappa_B)}\Bigr)
			= \frac{L_A^2}{4 \kappa_B(\kappa_A - \kappa_B)} \, ,
		\end{align*}
		as desired.
	\end{proof}
	
	\begin{remark}\label{rmk:change-behav-direct}
		Note that the right-hand side of \eqref{eq:two-cases} 
		involves the 
		the ratio of a ``directional Lipschitz constant'' and an 
		``effective convexity parameter''.
		In this sense, the bound has the same form as 
		\eqref{eq:transport-bound}.
		The bound \eqref{eq:two-cases} is sharp for 
		$\kappa_A \leq 2 \kappa_B$, 
		as we will see in the application to the Fisher model below.
	\end{remark}

	We finally state a corollary that will be used in the application to Fisher's infinitesimal model.
	Let $F = e^{-V}$ be a $\kappa$-log-concave 
	probability density on $\R^{2d}$ 
	for some $\kappa > 0$.
	For $x \in \R^d$ we consider the probability density
	$P(\cdot;x)$ on $\R^{2d}$ defined by 
	\begin{align}
		\label{eq:P-shape}
		P(x_1, x_2;x) 
		= \frac{1}{Z_x} 
		\exp\Bigl( - V(x_1, x_2) - \frac12 \Big| x - \frac{x_1 + x_2}{2}\Big|^2 \Bigr) \, ,		
	\end{align}
	where $Z_x > 0$ is the normalising constant which ensures that $P_x$ is a probability density.
	The transition rates appearing in the Fisher model are precisely of this form; see Theorem \ref{thm:main-W-infty}.

	\begin{corollary}
		\label{cor:main-W-infty}
		Let $F$ be a $\kappa$-log-concave 
		probability density on $\R^{2d}$
		for some $\kappa > \frac12$.
		Then, for any $x,\tilde x\in \R^d$,
		\begin{align}
			\label{eq:product-perturbation}
			W_{\infty}
				\bigl( 
				P(\cdot; x) , P(\cdot; \tilde x)
				\bigr) 
			\le 
			\frac{1}{\frac12 + \kappa} 
		\frac{| x - \tilde x|}{\sqrt{2}}
				\, .
		\end{align}
	\end{corollary}
	
	Before proving this result, 
	we first show that an application of the isotropic criterion from 
	Corollary 
	\ref{cor:log-lip-perturb-coupl}
	yields a suboptimal result.
	For ease of notation, suppose that $V \in C^2(\R^{2d})$.
	Fix $x, \tilde x \in \R^d$ 
	and let us write  
	$\mu_x = e^{-U} := P(\cdot;x)$ and 
	$\mu_{\tilde x} = e^{-H}\mu := P(\cdot; \tilde x)$.
        Then: 
	\begin{align*}
		\hess U(x_1, x_2) 
		= &
		\hess V(x_1, x_2) 
		+ \frac14 
		\begin{pmatrix}
			I_d & I_d
			\\
			I_d & I_d
		\end{pmatrix}
	\tand
		\nabla H(x_1, x_2)
		= \frac12
		\begin{pmatrix}
			x - \tilde x 
			\\
			x - \tilde x 
		\end{pmatrix} 
		\, .
	\end{align*}
	Taking into account that 
	$\hess V \succcurlyeq \kappa I_{2d}$ by assumption, 
	we have the bounds
	\begin{align}
		\label{eq:uniform-bounds}
		\hess U(x) \succcurlyeq  \kappa I_{2d}
		\tand
		|\nabla H(x)| 
			\leq 
		\frac{|x - \tilde x|}{\sqrt{2}}
		\, .
	\end{align}
	An application of 
		Corollary 
		\ref{cor:log-lip-perturb-coupl} 
	then yields the estimate
	$W_{\infty}(\mu_x, \mu_{\tilde x}) 
		\leq 
	\frac{|x - \tilde x|}{\kappa\sqrt{2}}$,
	which is weaker than the desired inequality \eqref{eq:product-perturbation}.
	(In particular, in the application to the Fisher model, where $\kappa = \beta$,
	the comparison of norms $|x|_1 \leq \sqrt{2}|x|_2$
	implies that 
	$W_{\infty,1}(\mu_x, \mu_{\tilde x}) 
	\leq 
	\frac{|x - \tilde x|}{\beta}$, 
	which is weaker than the desired inequality \eqref{eq:key-2}.)

	The following proof crucially exploits 
	anisotropy to obtain the sharp constant.

	\begin{proof}[Proof of Corollary \ref{cor:main-W-infty}]
	Consider the orthogonal decomposition of 
	$\R^{2d}$
	into symmetric and anti-symmetric vectors:
	$\R^{2d} = \R_{\rm s}^{2d} \oplus \R_{\rm a}^{2d}$, 
	where 
	\begin{align*}
		\R_{\rm s}^{2d} 
		:=
		\bigg\{ 
		\begin{pmatrix}
			x
			\\
			x
		\end{pmatrix}   \in \R^{2d}
		\ : \
		x\in \R^d
		\bigg\}
		\tand
		\R_{\rm a}^{2d} 
		:=
		\bigg\{ 
		\begin{pmatrix}
			x
			\\
			- x
		\end{pmatrix}  \in \R^{2d}
		\ : \
		x\in \R^d
		\bigg\}
		\, .
	\end{align*}
	The corresponding orthogonal projections
	$\sP_{\rm s}, \sP_{\rm a} : \R^{2d} \to \R^{2d}$ 
	have the form
	\begin{align*}
		\sP_{\rm s}
		\begin{pmatrix}
			x_1
			\\
			x_2
		\end{pmatrix} 
		=
		\frac{1}{2}
		\begin{pmatrix}
			x_1 + x_2
			\\
			x_1 + x_2
		\end{pmatrix}
		\tand 
		\sP_{\rm a}
		\begin{pmatrix}
			x_1
			\\
			x_2
		\end{pmatrix} 
		=
		\frac{1}{2}
		\begin{pmatrix}
			x_1 - x_2
			\\
			x_2 - x_1
		\end{pmatrix}
		\, .
	\end{align*}
	The crucial observation is now that the isotropic bounds \eqref{eq:uniform-bounds} can be replaced by more refined estimates that take into account how $U$ and $H$ behave in symmetric and anti-symmetric directions. 
	Namely, since 
	$
	\begin{psmallmatrix}
		I_d & I_d
		\\
		I_d & I_d
	\end{psmallmatrix} 
	= 2 \sP_{\rm s}
	$,
	we have the following improvement over \eqref{eq:uniform-bounds}:
	\begin{align*}
		\hess U(x) 
			\succcurlyeq
		\Bigl(\frac12 + \kappa\Bigr) \sP_{\rm s}  
		+ 
		\kappa \sP_{\rm a}
		\tand
		| \ip{\nabla H(x), z}|
		\leq \frac{| x - \tilde x|}{\sqrt{2}} |\sP_{\rm s} z | \,.
	\end{align*}
			(The first inequality holds when $V \in C^2(\R^{2d})$.
			In the general case, 
		the corresponding nonsmooth statement holds, 
		which asserts that 
		$U$ is 
		$K$-convex with 
		$K = \tond*{\frac12 + \kappa} \sP_{\rm s} 
		+ 
		\kappa \sP_{\rm a}$.)	
	Therefore, an application of 
	Corollary \ref{cor:directional}
	to 
	$A = \R^{2d}_{\rm s}$
	and 
	$B = \R^{2d}_{\rm a}$
	with parameters 
	\begin{align*}
		\kappa_A = \frac12 + \kappa
		\, , \quad
		\kappa_B = \kappa
		\, , \quad
		L_A =  \frac{| x - \tilde x|}{\sqrt{2}}
		\, ,
	\end{align*}
	yields, if 
	$\kappa \geq \frac12$,
	\begin{align*}
		W_{\infty}( \mu_x, \mu_{\tilde x} ) 
		\leq 
		\frac{1}{\frac12 + \kappa} 
		\frac{| x - \tilde x|}{\sqrt{2}}\, ,
	\end{align*}
	which is the desired inequality.
	\end{proof}

    \begin{remark}[Optimality]
		The constants in \eqref{eq:product-perturbation} are sharp.
		In fact, it was observed in \cite[Remark 2.7]{Calvez-Poyato-Santambrogio:2023} 
		that equality holds in the context of Fisher's infinitesimal model
		with quadratic selection in dimension $1$,
		which means that 
		$m(x) = \frac{\alpha}{2} x^2$
		with $\alpha > 0$. 
		In this case, we have
			$V(x_1,x_2) = \frac{\beta}{2}\tond*{x_1^2+x_2^2}$,
		with $\beta > \frac12$ as in Theorem \ref{thm:main-existence}.
		The measures 
			$P(\cdot  ; x)$ 
		are then Gaussian with mean 
			$(\frac12+\beta)^{-1}(\frac{x}{2},\frac{x}{2})$
		and 	
		the same covariance matrix. 
		The $W_\infty$-distance between two such measures is simply 
		the Euclidean distance between the respective means, which corresponds to the right-hand side in \eqref{eq:product-perturbation}.
		
		To show that this bound cannot be improved, take arbitrary densities $\mu$ and $\nu$ with finite first moment, and random variables 
		$X$ and $Y$ with marginals $\mu$ and $\nu$ respectively.
		Then:
		\begin{align*}
			\abs*{\int x \dd \mu(x) - \int x \dd\nu(x)} 
			= \bigl|  \expe {X - Y} \bigr|
			\leq \expe {| X - Y | }\, ,
		\end{align*} 
		which implies that, $\abs*{\int x \dd \mu(x) - \int x \dd\nu(x)}  \leq 
		W_\infty(\mu,\nu)$.
    \end{remark}

\subsection{Boundedness of the forward-flow transport map}

In this subsection we sketch an alternative argument to prove the transport bound of Corollary \ref{cor:log-lip-perturb-coupl}. Instead of constructing a suitable coupling, we provide an upper bound on the displacement of the \emph{forward-flow map}, whose inverse is
the so-called \emph{Langevin transport map}. 
The Langevin transport map and the forward-flow map were introduced by Kim and Milman \cite{Kim-Milman:2012} in their work on generalisations of Cafferelli's contraction theorem \cite{Caffarelli:2000}.
Subsequently, there has been a lot of interest in 
Lipschitz bounds for the forward-flow map 
\cite{mik-she-2023,Fathi-Mikulincer-Shenfeld:2023, nee-2022, kla-put-2021},
as such bounds allow one 
to transfer functional inequalities from log-concave measures to their image under the forward-flow map.
Here we show that $L^\infty$-bounds can be obtained as well.

As we already provided a rigorous proof of Corollary \ref{cor:log-lip-perturb-coupl} by a different method,
we keep the arguments in this section formal, 
so as not to obscure the main ideas.
In particular, we do not discuss the delicate issues of existence of flow maps. For more details on the construction and rigorous justifications we refer the reader to \cite{ott-vil-2000, Kim-Milman:2012, mik-she-2023, Fathi-Mikulincer-Shenfeld:2023}.

\paragraph*{Construction of the forward-flow map}

Consider probability densities 
	$\mu$ and $\nu$.
Here we assume that $\mu = e^{-U}$ and $\nu = e^{-H}\mu$   with smooth $U,H\colon\R^d\to \R $.
Moreover, $\mu$ is assumed to be $\kappa$-log-concave
(i.e.,~$\hess U \ge \kappa I_d$) for some $\kappa > 0$
and $\nu$ is a log-Lipschitz perturbation 
(i.e. $|\nabla H|\le L$ for some $L < \infty$).

	We shall briefly and informally describe the construction of the forward-flow map $S \colon \R^d \to \R^d$, 
which pushes-forward $\nu$ onto $\mu$
(i.e.,~$S_\# \nu = \mu$), referring the reader to the aforementioned references for details.

	The key idea is to interpolate between 
		$\nu$ and $\mu$
	using the Langevin dynamics 
	\[
		X_0 \sim \nu, 
		\qquad \dd X_t 
		= -\nabla U(X_t) \dd t 
		+ \sqrt{2} \dd B_t \, . 
	\]
	Denoting $\rho_t := \law (X_t)$, 
	we have $\rho_0 = \nu$ and 
	$\rho_t \to \mu$ weakly as $t \to \infty$. 
	Moreover, $\rho_t$ satisfies the Fokker-Planck equation, which we formulate here as a continuity equation
	\begin{equation}
		\partial_t \rho_t - \dive \bigl(\rho_t 
			\nabla \log f_t \bigr) = 0 \, ,
	\end{equation} 
	where $f_t := \frac{\ddd \rho_t}{\ddd \mu}$.
	Since $f_0 = \frac{\ddd \nu}{\ddd \mu} = e^{-H}$,
	our assumptions imply the pointwise bound
		$|\nabla \log f_0| 
		= |\nabla H| 
		\leq L$. 
	We will show that 
	\begin{align}
		\label{eq:nabla-log-f}
		|\nabla \log f_t| \leq L e^{- \kappa t}
	\end{align}	
	for all $t \geq 0$.

	For this purpose, let 
		$(P_t)_{t \geq 0}$ 
	be the transition semigroup associated to the Langevin dynamics, 
	and note that $f_t = P_t f_0$ by reversibility.
	Since $\mu$ is $\kappa$-log-concave, 
	the Bakry-\'Emery theory \cite[Thm.~3.3.18]{Bakry-Gentil-Ledoux:2014} 
	implies the pointwise gradient estimate 
	\begin{align}
		\label{eq:grad-est}
		| \nabla P_t f |
		\leq e^{-\kappa t} P_t |\nabla f| 
	\end{align}
	for all sufficiently regular $f : \R^d \to \R$.
	Using this inequality, the inequality 
	$|\nabla f_0| \leq L f_0$, and the positivity of $P_t$,
	we obtain
	\begin{align*}
		|\nabla f_t| 
		\leq 
		e^{-\kappa t} P_t |\nabla f_0|
		\leq 
		L e^{-\kappa t} P_t  f_0
		= 
		L e^{-\kappa t} f_t \, ,
	\end{align*}
	which yields the claimed bound \eqref{eq:nabla-log-f}.

        For $t \geq 0$, consider the flow map $S_t : \R^d \to \R^d$ 
	associated to the vector field 
		$(t,x) \mapsto - \nabla \log f_t(x)$, which satisfies
	\[
		S_0(x) = x \, , \qquad 
		\ddt S_t(x) = -\nabla \log f_t\big(S_t(x)\big) \, .
	\]  
	Then, by construction,
	$(S_t)_\# \nu = \rho_t$, 
	and for $0 \le s \le t$, \eqref{eq:nabla-log-f} yields
	\begin{equation}
		\label{eq:bound-displ-lang-flow}
		\| S_s - S_t \|_{L^\infty} 
		\leq
		\int_s^t
		 	\| \nabla \log (f_r \circ S_r) \|_{L^\infty} 
		\dd r
			\leq 
		\frac{L}{\kappa}
			\tond*{e^{-\kappa s} - e^{-\kappa t}} \, .
	\end{equation}
	Passing to the limit, it is now simple to deduce that 
	the forward-flow map
		$S = \lim_{t\to \infty} S_t$ 
	is well-defined, 
	that 
		$S_\#\nu = \mu$, 
	and that 
        \begin{equation}\label{eq:bound-flow-map}
         \| S - I\|_{L^\infty(\R^d)}\le \frac{L}{\kappa} \, .
         \end{equation}
	This is the desired bound,
	which immediately implies the bound $W_\infty(\mu,\nu)\le \frac{L}{\kappa}$ from \eqref{eq:transport-bound}.
 
	\smallskip 
	The inverse of the forward-flow map $S$ is known as the Langevin transport map. In general, these maps do not coincide with the Brenier map \cite{tat-2021, lav-san-2022}, except in dimension $1$. 
	An analogous bound to \eqref{eq:bound-flow-map} was proved in \cite[Prop. 3.1]{Calvez-Poyato-Santambrogio:2023} for the Brenier map, 
		under the stronger conditions that $\mu$ and $\nu$ are $\kappa$-log-concave, 
		supported on a Euclidean ball, and bounded away from $0$ on it. 
	It would be interesting to investigate whether \eqref{eq:bound-flow-map} can be improved in the presence of anisotropy, similarly to Theorems \ref{thm:main-abstract-coupl} and \ref{thm:main-directional}.

	\section{Applications to Fisher's infinitesimal model}
	\label{sec:application-fisher}
	
	Throughout this section, we fix $\alpha > 0$ 
	and an $\alpha$-convex mortality function
		$m \in C^1(\R^d)$. 
	We assume that $m \geq 0$ and $m(0) = 0$. 
	These assumptions are without loss of generality, except for the claim in Theorem \ref{thm:sol-trunc-prob} below that $\bflambda \in (0,1)$.
	We also fix $\beta > \frac12$ through the identity $\beta = \alpha + \frac{\beta}{\frac12 + \beta}$, as in Theorem \ref{thm:main-existence}.

        The following result is taken from 
	\cite[Lem.~2.4]{Calvez-Poyato-Santambrogio:2023}.
	For the convenience of the reader we include their proof.
	Recall that the metric $W_{\infty,1}$ was defined after \eqref{eq:key-1}.
	
	\begin{lemma}
	\label{lem:dual-inf-fish}
	Let $c > 0$, 
	and let 
		$P(\cdot \, ;\, x)$ 
	be a probability density on $\R^{2d}$
	for each $x \in \R^d$.
	Suppose that 
	 $u_{0}, u_1\in C(\R^d)$
	are strictly positive functions,
	that $\log u_0$ is $L$-Lipschitz,
	and that 
	\begin{equation}
		u_1({x}) = 
                c
			\int_{\mathbb{R}^{2d}} P({x}_1, {x}_2; x) 
			u_0({x}_1) 
			u_0({x}_2) 
			\dd {x}_1 \dd {x}_2 
	\end{equation}
	for all $x\in \R^d$.
	Then we have
  \begin{equation}
			|\log u_1({x}) - \log u_1 (\Tilde{{x}})|
				\leq 
			L \,
			W_{\infty, 1}\bigl( P(\cdot; {x}), P(\cdot; \Tilde{{x}})\bigr)
		\end{equation}
		for all ${x}, \Tilde{{x}} \in \mathbb{R}^d$.
	\end{lemma}

    \begin{proof}
        Fix $x,\tilde x \in \R^d$, and let  $\gamma \in \pp\tond*{\R^{2d}\times \R^{2d}}$ be an optimal coupling in the definition of 
			$W_{\infty,1}\bigl(
				P(\cdot; x), P(\cdot; \tilde x)
				\bigr)$.
		For $(x_1, x_2), (\tilde x_1, \tilde x_2) \in \R^{2d}$ 
		we have
		\begin{align*}
			& \log\bigl( 
				u_0({x}_1) u_0({x}_2) 
				\bigr) 
				- 
			\log\bigl(  
				u_0(\tilde{x}_1) 
					u_0(\tilde{x}_2) 
				\bigr)
			\\ & = 	
				\log u_0 ({x}_1) 
				- 
				\log u_0 (\tilde{{x}}_1) 
				+ 
				\log u_0 ({x}_2) 
				- 
				\log u_0 (\tilde{{x}}_2)
			 \leq  
				L  \bigl(
				|{x}_1 - \tilde{{x}}_1| 
			+ 	|{x}_2 - \tilde{{x}}_2|
				\bigr) \, .
		\end{align*}
        Writing $W :=W_{\infty, 1}
		\bigl(
		   P(\cdot;x), P(\cdot;\tilde x)
	   \bigr)$, it follows using the bound above
		that
		\begin{align*}
			u_1({x}) 
			& = 
                c
			\int_{\R^{4d}} 
			u_0({x}_1) u_0({x}_2) 
				\; \gamma
				(\ddd x_1, \ddd x_2, 
				\ddd \tilde x_1, \ddd \tilde x_2) 
			\\ & \leq 
                c
			\int_{\mathbb{R}^{4d}} 
			\exp\Bigl(
				L  \bigl(
				|{x}_1 - \tilde{{x}}_1| 
			+ 	|{x}_2 - \tilde{{x}}_2|
				\bigr)
				\Bigr)
			u_0(\tilde{{x}}_1) 
			u_0(\tilde{{x}}_2) 
				\; \gamma(\ddd {x}_1, \ddd {x}_2, \ddd \tilde{{x}}_1, \ddd \tilde{{x}}_2) 
			\\ & \leq 
				c 
				e^{
				L 
				W}
				 \int_{\mathbb{R}^{4d}} 
			u_0(\tilde{{x}}_1) 
			u_0(\tilde{{x}}_2) 
				\; \gamma(\ddd {x}_1, \ddd {x}_2, \ddd \tilde{{x}}_1, \ddd \tilde{{x}}_2) 
			= e^{LW}
				u_1(\tilde{{x}})
				\, .
		\end{align*}
	The desired conclusion follows after exchanging the roles of $x$ and $\tilde x$.
	\end{proof}

	\subsection{Analysis of a localised problem}
	
	As in \cite[Sec.~4]{Calvez-Poyato-Santambrogio:2023}, we study an auxiliary localised problem. 
	Specifically, for $R>0$, we consider the localised selection function
	\[
	m_R(x) \coloneqq m(x) + \chi_{\overline{B_R}} \, ,
	\]
	where $\chi$ denotes the convex indicator function,        
	i.e.,
	\[
	\chi_A (x) = \begin{cases}
		0 &\text{ if } x\in A \, ,
		\\
		+\infty &\text{otherwise}\, .
	\end{cases}
	\]
	The corresponding localised operator is given by
	\[
	\cT_R[F](x) = e^{-m_R(x)}   \int_{\R^{2d}} G\Big(x - \frac{x_1 + x_2}{2}\Big)  \frac{F(x_1) F(x_2)}{\|F\|_{L^1}} \dd{x_1} \dd{x_2}\, .
	\]
	In this section we establish the existence of a quasi-stationary distribution for the localised problem, adapting the proof of 
		\cite[Thm. 4.1(i)]{Calvez-Poyato-Santambrogio:2023}.
	
	\begin{theorem}\label{thm:sol-trunc-prob}
		Let $R > 0$. 
		There exists $\bflambda_R \in (0,1)$ and a $\beta$-log-concave probability density $\bfF_R$ on $\R^d$ 
		that 
		is 
		bounded away from $0$ on its support
		$\overline{B_R}$,
		and 
		satisfies
		\[
		\cT_R[\bfF_R] = \bflambda_R \bfF_R 
		\, .
		\]  
	\end{theorem}

        The following result, proved in \cite[Lem.~2.2, 2.3]{Calvez-Poyato-Santambrogio:2023}, is an immediate consequence of the fact that log-concavity is preserved by convolution (Lemma \ref{lem:convolution}) and pointwise multiplication with log-concave functions.
	\begin{lemma}[Preservation of log-concavity]\label{lem:log-conc-prop}
		Let $R > 0$.
		If $F$ is $\kappa$-log-concave for some $\kappa > 0$, 
		then 
		$\cT[F]$ and $\cT_R[F]$ are $\kappa'$-log-concave
		with 
		$\kappa' := 
			\alpha + \frac{2\kappa}{1+2\kappa}$.
		In particular, 
		if $F$ is $\beta$-log-concave, then
		$\cT[F]$ and $\cT_R[F]$ are $\beta$-log-concave as well.
	\end{lemma}

	The key ingredient in the proof of 
		Theorem \ref{thm:sol-trunc-prob} 
	is the following contractivity estimate.
	
	\begin{lemma}
		\label{lem:cauchy-inf-fish}
		Define $F_0 \in L^1_+(\R^d)$ by 
		$F_0(x) = 
		\exp
		\bigl(- \frac{\beta}{2}\norma{x}^2
		- \chi_{\overline{B}_R(0)}(x)
		\bigr)
		$
		and set $F_{n + 1} \coloneqq T_R[F_{n}]$ for $n \geq 0$. 
		Then, for all $n \geq 1$:
		\[
		\fisinf{F_{n+1}}{F_{n}} 
		\leq 
		\bigl(\tfrac12 + \beta\bigr)^{-1}
		\fisinf{F_{n}}{F_{n-1}} 
		\, .
		\]
	\end{lemma}

        \begin{proof}
		Set $B_R := {B}_R(0)$ for brevity,
		and 
		define $u_n \coloneqq \frac{F_n}{ F_{n-1}}$ for 
			$n\ge 1$.     
         Note that $u_n$ is of class $C^1$ on $\overline{B_R}$; moreover, it is strictly positive there, since so is $F_n$ (by induction, using that the integral of a strictly positive function on a set of strictly positive measure is also strictly positive).
		Using the identities
		\begin{align*}
            F_{n}(x) & =   \frac{e^{-m_R(x)}}{\normad{F_{n-1}}_{L^1}} \int_{\R^{2d}} F_{n-1}(x_1) F_{n-1}(x_2)  
			G\Bigl(x - \frac{x_1 + x_2}{2}\Bigr) \dd x_1 \dd x_2\, , \\
			F_{n+1}(x)
            & =  \frac{e^{-m_R(x)}}{\normad{F_n}_{L^1}} \int_{\R^{2d}} F_{n-1}(x_1)u_n(x_1) F_{n-1}(x_2) u_n(x_2) 
			G\Bigl(x - \frac{x_1 + x_2}{2}\Bigr)\dd x_1 \dd x_2 \, , 
        \end{align*}
		we obtain the recursion relation
		\begin{equation}
			u_{n+1}(x) 
			= 
			 \frac{\normad{F_{n-1}}_{L^1}}{\normad*{F_{n}}_{L^1}}
			\int_{\R^{2d}} 
			P_n(x_1,x_2;x) u_{n}(x_1) u_{n}(x_2) 
			\dd x_1 \dd x_2 
		\end{equation}
		for $x\in \clobal$ and $n\ge 1$,
		with $n$-dependent transition rates 
		\[
		P_n(x_1,x_2;x) 
		= \frac{1}{Z_{n}(x)}  
		{F_{n-1}(x_1) F_{n-1}(x_2) }
		G\Big(x - \frac{x_1 + x_2}{2}\Big) \,,
		\]
        where $Z_{n}(x) > 0$ is the normalising constant
		ensuring that $P_n(\cdot \, ;  x)$
			is a probability density for all $x \in \R^d$.
		Arguing as in the proof of 
			Lemma \ref{lem:dual-inf-fish},
		we infer that 
		\begin{align*}
			|\log u_{n+1}({x}) 
				- 
			\log u_{n+1} (\Tilde{{x}})|
			\leq 
		\| \nabla \log u_n \|_{L^\infty(\overline B_R)} \,
		W_{\infty, 1}\bigl( P_n(\cdot; {x}), P_n(\cdot; \Tilde{{x}})\bigr)
	\end{align*}
	for all ${x}, \Tilde{{x}} \in 
		\overline{B_R}$.
	Since
	$F_n$ is $\beta$-log-concave 
	by Lemma \ref{lem:log-conc-prop}, 
	Corollary \ref{cor:main-W-infty}
	yields, 
	in view of the elementary 
	comparison of norms
		$|(x_1, x_2)|_1 \leq \sqrt{2} |(x_1, x_2)|$ 
	for $x_1, x_2 \in \R^d$,
	\begin{align*}
		W_{\infty,1}
				\bigl( 
				P_n(\cdot; x) , P_n(\cdot; \tilde x)
				\bigr) 
				\leq 
			\sqrt{2}
			W_{\infty}
				\bigl( 
				P_n(\cdot; x) , P_n(\cdot; \tilde x)
				\bigr) 
			\le 
			\frac{| x - \tilde x|}{\frac12 + \beta}  \, .
	\end{align*}
	Combining these inequalities, we find
	\begin{align*}
		\| \nabla \log u_{n+1} \|_{L^\infty(\overline B_R)}
		\leq \bigl(\tfrac12 + \beta\bigr)^{-1}
		\| \nabla \log u_n\|_{L^\infty(\overline B_R)}\, ,
	\end{align*}
	which is the desired inequality.
\end{proof}

	\begin{proof}[Proof of Theorem \ref{thm:sol-trunc-prob}]
		
		Set $F_0 =  \frac{1}{Z} \exp
		\bigl(- \frac{\beta}{2}\norma{x}^2
		- \chi_{\overline{B_R}}
		\bigr)$ as in Lemma \ref{lem:cauchy-inf-fish}
		and define 
		$F_{n+1} = \cT_R [F_n]$
		for $n \geq 0$,
		and 
		write $V_n := - \log F_n$.
		Clearly, the restriction of $F_n$ to $\overline{{B_R}}$ 
		(which will simply be denoted by $F_n$ as well)
		is bounded away from $0$ and it
		belongs to 
		$C^1(\overline{{B_R}})$ for all $n \geq 0$.
		Adapting arguments from 
		\cite{Calvez-Poyato-Santambrogio:2023}, 
		we will show that 
		$\log \bigl(  F_n / \| F_n\|_{L^1}  \bigr)$ 
		converges in $C(\overline{{B_R}})$ as $n \to \infty$.
		This statement will follow from two claims.		
		
		Firstly, we claim that 
		$\nabla \log \bigl(  F_n / \| F_n\|_{L^1}  \bigr) =  -\nabla V_n$ 
		converges in $C(\overline{{B_R}})$ as $n \to \infty$.	 
		To prove this, we observe that 
		Lemma \ref{lem:cauchy-inf-fish} 
		yields
		\begin{align*}
			\fisinf{F_{n+1}}{F_{n}} 
			\leq 
			\bigl(\tfrac12 + \beta\bigr)^{-n}\,
			\fisinf{F_{1}}{F_{0}} \, .
		\end{align*}
		Since 
		$\fisinf{F_{n+1}}{F_{n}}
		= 
		\| 
		\nabla V_n -  \nabla V_{n+1} 
		\|_{C(\overline{{B_R}})}$
		,
		the sequence $\nabla V_n$ is Cauchy in $C(\overline{{B_R}})$, hence convergent. 
		
		Secondly, we claim that 
			$\frac{F_n(0)}{\|F_n\|_{L^1}}$ 
		converges in $\R$ as $n \to \infty$. 
		To show this, we use the identity
		\begin{equation*}
			\frac{F_n(x)}{\|F_n\|_{L^1}}
			=
			\frac{
				\int_{\qbr} G\tond*{x - \frac{x_1+x_2}{2}} 
				\exp 	  \bigl(- m(x) 
				- v_{n}(x_1)- v_{n}(x_2) \bigr) \dd x_1 \dd x_2 }
			{\iint_{\qbr\times \clobal}G\tond*{x'-\frac{x_1+x_2}{2} } \exp   \bigl(
				-m(x')
				- v_{n}(x_1)- v_{n}(x_2)
				\bigr)
				\dd x_1 \dd x_2 \dd x' } \, ,
		\end{equation*}
		where we write $v_n(x) = V_{n-1}(x) - V_{n-1}(0)$ for brevity. 
		Note that the artifically introduced factors $e^{V_{n-1}(0)}$ cancel out. 
		Writing
		$v_{n}(x) = x\cdot \int_0^1 \nabla V_{n-1}(\theta x) \dd \theta$ 
		we infer from the first claim that $v_{n}$ converges uniformly. 
		Therefore, the second claim follows using dominated convergence.
		
		The two claims combined imply that 
		$\log \bigl(  F_n / \| F_n\|_{L^1}  \bigr)$ 
		converges in $C(\overline{{B_R}})$ as $n \to \infty$.
		Let $-\bfV_R$ be its limit, and define 
		$\bfF_R := \exp\bigl(-\bfV_R - \chi_{\overline{B_R}}\bigr)$.
		It remains to verify that $\bfF_R$ has the desired properties.
		
		Since $\bfV_R$ is bounded, it follows that $\bfF_R$ is bounded away from $0$ on its support $\overline{B_R}$.
		
		To prove the identity  
		$
		\cT_R\quadr*{\bfF} = \bflambda_R
		\bfF_R$ 
		we write
		\[
		\frac{ \| F_{n+1} \|_{L^1} }
			 { \| F_{n}   \|_{L^1} }  
		= 
		\int_{\qbr} H_R(x_1,x_2) 
		\frac{F_{n}(x_1)}{\|F_{n}\|_{L^1}} 
		\frac{F_{n}(x_2)}{\|F_{n}\|_{L^1}} \dd x_1 \dd x_2 \, ,
		\]
		where 
		$H(x_1,x_2) 
			= 
			\int_{\clobal} 
				e^{-m(x)}
				G\bigl(x-\frac{x_1+x_2}{2}\bigr) 
			\dd x$
		is bounded.
		Since $H_R$ is bounded and 
			${F_{n}}/{\|F_{n}\|_{L^1}}$ 
		converges uniformly by the first part of the proof,
		we infer that 
		$\frac{ \| F_{n+1} \|_{L^1} }
			  { \| F_{n}   \|_{L^1} }  
			\to \bflambda_R$ 
		for some 
			$\bflambda_R > 0$.
		Since $\frac{F_n}{\|F_n\|_{L^1}} \to \bfF$ 
		in $C(\overline{B_R})$, 
		it follows that 
			$\cT_R\quadr*{ \frac{F_n}{\|F_n\|_{L^1}} }
				\to 
			\cT_R\quadr*{\bfF}$.
		On the other hand,
		\[
		\cT_R\quadr*{ \frac{F_n}{\|F_n\|_{L^1}} }
		= 
		\frac{ \| F_{n+1} \|_{L^1}}
			 { \|F_{n}    \|_{L^1}}
		\frac{F_{n+1}}{\|F_{n+1}\|_{L^1}} 
		\to
		\bflambda_R
		\bfF_R
		\]
		as $n \to \infty$.
		This yields the desired identity 
		$
		\cT_R \quadr*{\bfF} 
		= 
		\bflambda_R \bfF_R$.
		
		\smallskip
		Since $F_0$ is a $\beta$-log-concave density, 
		so 
		are all $F_n$ by Lemma \ref{lem:log-conc-prop}.
		Therefore, the functions 
		$-\log \bigl(  F_n / \| F_n\|_{L^1}  \bigr)$ 
		are $\beta$-convex, 
		and so is their uniform limit $\bfV_R$.
		It follows that
		$\bfF_R$ is $\beta$-log-concave.

		\smallskip
		Finally, we will show that 
		$\bflambda_R \in (0,1)$. 
		Indeed, 
		since $\bfF_R$ is quasi-stationary, 
		we have
            \[
                \bflambda_R \bfF_R(x) e^{m_R(x)} =   \int_{\R^{2d}} G\Big(x - \frac{x_1 + x_2}{2}\Big)  \bfF_R( x_1) \bfF_R({x_2}) \dd{x_1} \dd{x_2}\, .
            \]
        From this, 
		it is immediate to see that $\bflambda_R>0$, by choosing $x=0$.
            To see that $\bflambda_R<1$, it suffices to integrate over $x\in \R^d$ on both sides.
            Indeed, there exists a small $\delta\in(0,R)$ such that 
			$c_\delta := \int_{B_\delta(0)} \bfF_R(x) \dd x <1$. But then, using the assumptions on $m$, 
            \begin{align*}
                \int_{ \R^d} e^{m_R(x)}\bfF_R(x) \dd x 
                & \ge  
                    e^{{\alpha}\delta^2/{2}}\int_{B_\delta(0)} \bfF_R(x) \dd x + \int_{B_\delta(0)^c}\bfF_R(x) \dd x 
				\\& = 
                    c_\delta e^{{\alpha}\delta^2/{2}} + (1-c_\delta) > 1 \,,
            \end{align*}
            while
            \[
                \int_{\R^d} \int_{\R^{2d}} G\Big(x - \frac{x_1 + x_2}{2}\Big)  \bfF_R( x_1) \bfF_R({x_2}) \dd {x_1} \dd {x_2}\, \dd x = 1
				\, .
            \]
			Consequently, $\bflambda_R \in (0,1)$. 
	\end{proof}
	
	\subsection{Existence of a \texorpdfstring{$\beta$}{beta}-log-concave quasi-equilibrium}
 
	The following second-moment bound is an analogue of \cite[Prop. 5.1]{Calvez-Poyato-Santambrogio:2023}, but the proof is based on different arguments that seem more convenient in the multi-dimensional setting.
	In particular, we use various properties of maxima of strongly log-concave densities, that are proved in Section \ref{sec:peaks-log-conc}.

	\begin{proposition}
		\label{prop:second-moment-bound}
		For $R > 0$, let $F_R$ be a solution to the localised problem given in Theorem \ref{thm:sol-trunc-prob}.
		Then:  
		\[
		\sup_{R>0} 
			\int_{\R^d} |x|^2 F_R(x) \dd x 
			< \infty \, .
		\]
	\end{proposition}
	
	\begin{proof}
		Let $\mu_R = \int_{\R^d} x \bfF_R(x) \dd x$ 
		be the barycenter of $\bfF_R = e^{-\bfV_R}$.
		Since the measures $\bfF_R$ are $\beta$-log-concave,
		it follows from the Poincar\'e inequality
		\cite[Prop.~4.8.1]{Bakry-Gentil-Ledoux:2014}
		that
		\begin{align*}
			\sup_{R>0} \int_{\R^d} |x - \mu_R|^2 F_R(x) \dd x
			\leq \frac{d}{\beta} <\infty \, .
		\end{align*}
		Therefore, it suffices to show that 
		$\sup_{R > 0} |\mu_R| < \infty$.
		Let $v_R \in \R^d$ be the unique minimiser 
		of $ \bfV_R$. 
		Since $\bfF_R$ is $\beta$-log-concave,
		Lemma 
		\ref{lem:conc-around-argmin} 
		implies that
		\[
		|v_R - \mu_R|  \le \sqrt{\frac{d}{\beta}} \, .
		\]
		
		Define $G_R := \cR[\bfF_R]$ and write $G_R = e^{-U_R}$. 
		The barycenter $\mu_R$ of
		$\bfF_R$ is also the barycenter of $G_R$, 
		since 
		$G_R$ can be written in probabilistic terms 
		as
		$G_R = \law\bigl(\frac{X_R+\tilde X_R}{2} + Z\bigr)$,
		where $X_R, \tilde X_R$ are independent random variables with law $F_R$ and  
		$Z$ is standard Gaussian, and we have
		$\E\bigl[\frac{X_R+\tilde X_R}{2} + Z\bigr] 
		= \E[ X_R]$.
		Moreover, 
		Lemma \ref{lem:log-conc-prop} implies that 
		$G_R$ is $\tau$-log-concave with 
		$\tau := \beta/(\frac12 + \beta)$.
		Therefore,
		another application of Lemma 
		\ref{lem:conc-around-argmin} yields
		\begin{align}
			\label{eq:mu-u}
			|u_R - \mu_R|  \le \sqrt{\frac{d}{\tau}} \, ,
		\end{align}
		where $u_R \in \R^d$ denotes the unique minimiser of $U_R$.
		
		Since
		$\cT[\bfF_R] = \bflambda \bfF_R$,
		it follows that 
		$\bfV_R = m_R + U_R + \log \bflambda_R$.
		Recall that 
			$m_R$ has its unique minimiser at $0$
		and satisfies $\hess m_R \succcurlyeq \alpha I_d$.
		Observe that Lemma \ref{lem:upp-bound-hess-pot-gauss-conv}
		implies that 
		$\hess U_R \preccurlyeq I_d$.
		Therefore, 
			Lemma \ref{lem:min-sum-conve-and-con-smooth}
		implies that 
		\begin{align*}
			{\alpha} |u_R  | \leq (1+\alpha)  | u_R - v_R | \, .
		\end{align*}
		
		Combining the three inequalities above, we find
		\begin{align*}
			\frac{\alpha}{1+\alpha}  |u_R| 
			\leq 
			| u_R - v_R | 
			& \leq 
			| u_R - \mu_R | + |  v_R - \mu_R |
			\leq 
			\sqrt{\frac{d}{\tau}}  
			+ \sqrt{\frac{d}{\beta}}  \, .
		\end{align*}
		Another application of \eqref{eq:mu-u} implies that $\sup_{R > 0} |\mu_R| < \infty$, as desired.
	\end{proof}

To prove Theorem \ref{thm:main-existence},
 we can now
follow the argument from 
\cite[Thm. 5.2]{Calvez-Poyato-Santambrogio:2023}.

\begin{proof}[Proof of Theorem \ref{thm:main-existence}]
	It follows from Proposition \ref{prop:second-moment-bound} that 
	the family of probability measures 
			$\{ \bfF_R \}_{R > 0}$ 
	is tight. 
	Therefore, there exists a sequence of radii $(R_n)_n$ with $R_n\uparrow \infty$ and a limiting probability measure $\bfF$ such that
        $\bfF_{R_n}\to \bfF$ weakly.
	Then, proceeding as in the proof of \cite[Thm. 5.2]{Calvez-Poyato-Santambrogio:2023}, it follows that $\bflambda_{R_n}$ converges to some $\bflambda\in [0,1]$, that $\bfF$ is $\beta$-log-concave, and that the pair $(\bflambda,\bfF)$ satisfies 	$\cT [\bfF] = \bflambda \bfF$. 
	Proceeding as in the proof of Theorem \ref{thm:sol-trunc-prob} we also find that $\bflambda\in(0,1)$.
\end{proof}

\subsection{Exponential convergence to quasi-equilibrium}

\begin{proof}[Proof of Theorem \ref{thm:main-W-infty}]
	Recall from \eqref{eq:P-def} that 
	\begin{align}
		P(x_1, x_2;x) 
		= \frac{1}{Z_x} 
		\bfF(x_1) \bfF(x_2)
		\exp\Bigl(
			- \frac12 \Big| x - \frac{x_1 + x_2}{2}\Big|^2 
			\Bigr) \, ,		
	\end{align}
	where $\bfF$ is a
	$\beta$-log-concave  quasi-equilibrium 
	obtained in Theorem \ref{thm:main-existence}.
	Therefore, the result follows from
	Corollary \ref{cor:main-W-infty}.
\end{proof}

\begin{proof}[Proof of Corollary \ref{cor:main-conv-fisher-model}]

We follow the proof of \cite[Thm.~1.1]{Calvez-Poyato-Santambrogio:2023}; for the convenience of the reader, we reproduce the argument here.
	
\smallskip

	\eqref{it:fis-inf-contr}: \ 
	Take $0 \neq F_0 \in L^1_+(\R^d)$ with 
		$\fisinf{F_0}{\bfF} < \infty$. 
	Then we can write 
		$F_0 = u_0 \bfF$ 
	for some strictly positive 
		$u_0\in C(\R^d)$ 
	such that $\log u_0$ is $L$-Lipschitz with 
		$L: = \fisinf{F_0}{\bfF}$.
	For $n \ge 1$, 
	set 
		$u_n = \frac{F_n}{\bflambda^n \bfF}$. 
	We will show by induction that 
    	$\log u_n$ is $L_n$-Lipschitz
	with  
		$L_n = \tond*{\frac12+\beta}^{-n}\fisinf{F_0}{\bfF}$, 
	which implies \eqref{it:fis-inf-contr} in Corollary \ref{cor:main-conv-fisher-model}.
    To this end, recall that we have the recursion
     \begin{align*}
		u_{n+1}(x)
		=
		\int_{\R^{d} \times \R^{d}} 
		\frac{u_n(x_1) u_n(x_2)}{\|u_n \bfF\|_{L^1}}
		P(x_1, x_2;x)
		\dd x_1 \dd x_2
	\end{align*}
	for all $x \in \R^d$.
   Therefore, Lemma \ref{lem:dual-inf-fish} implies
        \begin{align*}
            \norma*{\log u_{n+1}(x_1)-\log u_{n+1}(x_2)} 
			\le 
			\fisinf{F_n}{\bfF}\,
			W_{\infty,1}
			\bigl( P(\cdot; {x}), P(\cdot; \Tilde{{x}})\bigr)
        \end{align*}
	for all $x_1, x_2\in \R^d$ with $x_1 \neq x_2$.
	Using the 
	elementary bound 
		$W_{\infty,1}\le \sqrt{2}W_{\infty}$
	and Theorem \ref{thm:main-W-infty}, we obtain
		\begin{align*}
			W_{\infty,1}
				\bigl( 
				P(\cdot; x) , P(\cdot; \tilde x)
				\bigr) 
				\leq 
			\sqrt{2}
			W_{\infty}
				\bigl( 
				P(\cdot; x) , P(\cdot; \tilde x)
				\bigr) 
			\le 
			\frac{| x - \tilde x|}{\frac12 + \beta}  
			\, .
        \end{align*}
		Combining these inequalities, we find
		$\fisinf{F_{n+1}}{\bfF} 
		\leq
		\big(\tfrac12 + \beta\big)^{-1}
	\fisinf{F_n}{\bfF}$,
        which implies the desired conclusion. 
        
		\smallskip

		\eqref{it:directional-Lip}: \
		For brevity, write 
			$\widehat F_n := {F_n}/{\| F_n \|_{L^1}}$.
        Since $\bfF$ is $\beta$-log-concave, 
		it satisfies a logarithmic Sobolev inequality by the Bakry--\'Emery theory \cite[Cor. 5.7.2]{Bakry-Gentil-Ledoux:2014}).
		Using this 
		and the trivial bound 
			$\fistwo{\cdot}{\bfF}
			\le 
			\fisinf{\cdot}{\bfF}^2$, 
		we deduce that
        \[
            \kldiv{ \widehat F_n}
					   {\bfF} 
			\le 
				\frac{1}{2\beta} 
					\fistwo{
						\widehat F_n}
						{\bfF} 
			\le
				\frac{1}{2\beta} 
					\fisinf{F_n}{\bfF}^2
			\le 
				\frac{\fisinf{F_0}{\bfF}^2}{2\beta\bigl(\tfrac12+\beta\bigr)^{2n}} 
				\, .
        \]
        As for the last conclusion, set $\phi := e^{-m}*G \in C_b(\R^d)$,
		and note that
        \begin{align*}
            \frac{ \| F_{n+1} \|_{L^1} }
				{ \| F_n \|_{L^1} } 
			& = 
			\int_{\R^{2d}} 
			\phi\Bigl( \frac{x_1+x_2}{2} \Bigr)
			\widehat F_n(x_1)
			\widehat F_n(x_2)
				\dd x_1 \dd x_2 \, ,
        \\
        \bflambda & = 			
			\int_{\R^{2d}}
			\phi\Bigl( \frac{x_1+x_2}{2} \Bigr) 
				\bfF(x_1) 
				\bfF(x_2) 
					\dd x_1 \dd x_2	\,,
        \end{align*}
		hence, by H\"older's inequality,
		\begin{align*}
			\abs*{ \frac{ \| F_{n+1} \|_{L^1} }{ \| F_n \|_{L^1} } 
		- \bflambda } 
		& \le 
			\|\phi\|_{L^\infty} 
			\normad*{
				\widehat F_n 
				\otimes 
				\widehat F_n 
				- \bfF \otimes \bfF 
				}_{L^1} \, .
		\end{align*}
	Using Pinsker's inequality, the tensorization of the relative entropy, and the previous step we deduce that
        \begin{align*}
			\normad*{
				\widehat F_n 
				\otimes 
				\widehat F_n 
				- \bfF \otimes \bfF 
				}_{L^1}
        &  \le  
		\sqrt{\frac12 
		\kldiv{ \widehat F_n \otimes \widehat F_n }
				{\bfF \otimes \bfF}}
         \\& \le 
		\sqrt{\kldiv{ \widehat F_n }{\bfF}}
         \le 
		\frac{\fisinf{F_0}{\bfF} }{\sqrt{2\beta}\bigl(\tfrac12+\beta\bigr)^{n}} 
		\, ,
        \end{align*}
        which gives the desired conclusion.
\end{proof}

 \section{Other information metrics}
 \label{sec:different_metrics}

 In view of the  contraction estimate 
 	$\fisinf{\cT[F]}{\bfF}
		\leq 
		\big(\tfrac12 + \beta\big)^{-1}
	\fisinf{F}{\bfF}
	$,
it is natural to ask whether analogous inequalities
hold for other functionals $\cF$, such as the relative entropy $\kldiv{\cdot}{\bfF}$ and the $L^2$ relative Fisher information 
$\fistwo{\cdot}{\bfF}$, 
which play a central role in the 
	Bakry--\'Emery theory for diffusion equations.

Here we consider the case of quadratic selection 
	$m(x) = \frac{\alpha}{2}|x|^2$ for some $\alpha > 0$
in dimension $d = 1$,
which has been analysed in detail in \cite{Calvez-Lepoutre-Poyato:2021}.
In this case, the operator $\cT$ maps Gaussian densities to multiples of Gaussian densities. 
Indeed, 
for $\mu \in \R$ and $\sigma > 0$, we have
\begin{align}
	\label{eq:recursion-rel}
	\cT[\gamma_{\mu, \sigma^2}]
		\propto 
		\gamma_{\tilde \mu, \tilde \sigma^2}\, , 
	\quad \text{ where } \
	\tilde \mu 
		= 
		\frac{\mu}{1 + \alpha \bigl(1 + \frac{\sigma^2}{2}\bigr)} 
	\tand 
	\tilde \sigma^2 
		= 
		\frac{1 + \frac{\sigma^2}{2}}{1 
			+ \alpha \bigl(1 + \frac{\sigma^2}{2}\bigr)} \, ,
\end{align}
see (3.5) in \cite{Calvez-Lepoutre-Poyato:2021}.
Moreover, the unique quasi-stationary probability distribution $\bfF$ is the centered Gaussian density with variance $\frac{1}{\beta}$,
where $\beta >\frac12$ is the log-concavity parameter in Theorem \ref{thm:main-existence}; see (1.12) and (1.13) in \cite{Calvez-Lepoutre-Poyato:2021}.  

To analyse the behaviour of the three functionals under $\cT$, 
we consider the renormalised operator
$\widehat\cT$ given by 
	$\hcT[F] := \cT[F] / \|\cT[F]\|_{L^1}$
that preserves probability densities.
Let us first consider the case where 
	$G_\mu := \gamma_{\mu, \frac1{\beta}}$ is 
	a Gaussian density 
	having the variance $\frac1\beta$ of the quasi-equilibrium $\bfF$
	with arbitrary nonzero mean $\mu \in \R$. 
Then $\hcT[G_\mu]$ is Gaussian with variance $\frac1\beta$ as well, and the three functionals contract with the same rate:
\begin{align}
	\label{eqn:mu_n_sigma_n}
	\frac{\kldiv{\hcT[G_\mu]}{\bfF}}{\kldiv{G_\mu}{\bfF}}		
	= 
	\frac{\fistwo{\hcT[G_\mu]}{\bfF}}{\fistwo{G_\mu}{\bfF}}
	= 
	\Biggl(\frac{\fisinf{\hcT[G_\mu]}{\bfF}}{\fisinf{G_\mu}{\bfF}}
	\Biggr)^2
	= \bigl( \tfrac{1}{2} + \beta \bigr)^{-2} 	< 1	\, .
\end{align}
These equalities readily follow from the following Gaussian identities,
which hold for $\mu, \bar \mu \in \R$ and $\sigma^2, \bar \sigma^2 > 0$:
\begin{align}
	\kldiv{
		\gamma_{\mu, \sigma^2}}
		{\gamma_{\bar \mu, \bar \sigma^2}}
			& = 
			\frac12\biggl(
				\frac{
				(\mu - \bar \mu)^2}{\bar \sigma^2}
			+ \log\tond*{\frac{\bar \sigma^2}{\sigma^2}} 
			- 1
			+ 
				\frac{\sigma^2}{\bar \sigma^2} 
			 \biggr)
			\, ,
			\\
			\fistwo{
				\gamma_{\mu, \sigma^2}}
				{\gamma_{\bar \mu, \bar \sigma^2}}
				& = 
				\frac{(\mu - \bar \mu)^2 
					}
						{\bar \sigma^4}
					+ \frac{
					 (\sigma^2 - \bar \sigma^2)^2}
						{\sigma^2 \bar \sigma^4} \, ,
				\\
    \fisinf{
			\gamma_{\mu, \sigma^2}}
			{\gamma_{\bar \mu, \bar \sigma^2}}
	& = \frac{| \mu - \bar \mu |}{\sigma^2} 
		\quad \text{if } 
			\sigma = \bar \sigma;  \quad 
		\text{otherwise, } \
		\fisinf{\gamma_{\mu, \sigma^2}}
		{\gamma_{\bar \mu, \bar \sigma^2}}  = + \infty \; . 
\end{align}

Next, let us suppose that $G = \gamma_{\mu, \sigma^2}$ is a Gaussian density with arbitrary mean $\mu \in \R$ and 
variance $\sigma^2 \neq \frac1\beta$.
In this case, 
$\fisinf{G}{\bfF} 
= \fisinf{T[G]}{\bfF}
= + \infty$. However, the relative entropy and the $L^2$ relative Fisher information are finite, so
one might wonder whether these functionals contract under $\hcT$
with the rate suggested by \eqref{eqn:mu_n_sigma_n}.
The following result shows that this is not the case.

\begin{proposition}
	Let $m(x) = \frac{\alpha}{2}|x|^2$ for some $\alpha > 0$, 
	and define 
	$\beta > \max \bigl\{ \frac12, \alpha \bigr\}$ 
	by $\beta = \alpha + \frac{\beta}{\frac12 + \beta}$, 
	as in Theorem \ref{thm:main-existence}.
	Then there exist Gaussian probability densities 
		$G \in L_+^1(\R)$
	such that
	\begin{align*}
			\frac{\fistwo{\hcT[G]}{\bfF}}
				{\fistwo{G}{\bfF}}
				> \bigl( \tfrac{1}{2} + \beta \bigr)^{-2} 
		\tand
			\frac{\kldiv{\hcT[G]}{\bfF}}
				{\kldiv{G}{\bfF}}
				> \bigl( \tfrac{1}{2} + \beta \bigr)^{-2} 
		\, .
	\end{align*}		
\end{proposition}

\begin{proof}
	Let $\cF$ be either 
	$\kldiv{\cdot}{\bfF}$ 
	or 
	$\fistwo{\cdot}{\bfF}$.
	Using \eqref{eq:recursion-rel} 
	we observe that 
	\begin{align*}
		\lim_{\mu \to \infty}
		\frac{\cF \bigl(\hcT[\gamma_{\mu, \sigma^2}]\bigr) }{\cF (\gamma_{\mu, \sigma^2})}		
		=  \Bigl( 
			1 + \alpha
			\bigl(
				1 + \tfrac{\sigma^2}{2}
			\bigr)
			\Bigr)^{-2} \, . 
	\end{align*}
	Consequently, 
	\begin{align}
		\label{eq:cont-fac}
		\lim_{\sigma^2 \to 0}
		\lim_{\mu \to \infty}
		\frac{\cF \bigl(\hcT[\gamma_{\mu, \sigma^2}]\bigr) }{\cF (\gamma_{\mu, \sigma^2})}		
		=  ( 
			1 + \alpha
			)^{-2} \, . 
	\end{align}
	Since
		$(1 + \alpha)^{-2}
			 > 
		(\frac{1}{2} + \beta)^{-2}$,
	the claim follows.
\end{proof}

We illustrate the behaviour of the contraction factor
	$ \displaystyle C_\cF(\mu, \sigma^2) := 
	\frac{\cF \bigl(\hcT[\gamma_{\mu, \sigma^2}]\bigr) }{\cF (\gamma_{\mu, \sigma^2})}$ 
	in Fig. \ref{fig:contraction}.

\begin{figure}[ht]
    \centering
    \begin{subfigure}[t]{0.45\columnwidth}
    \centering
    \includegraphics[width=\columnwidth]{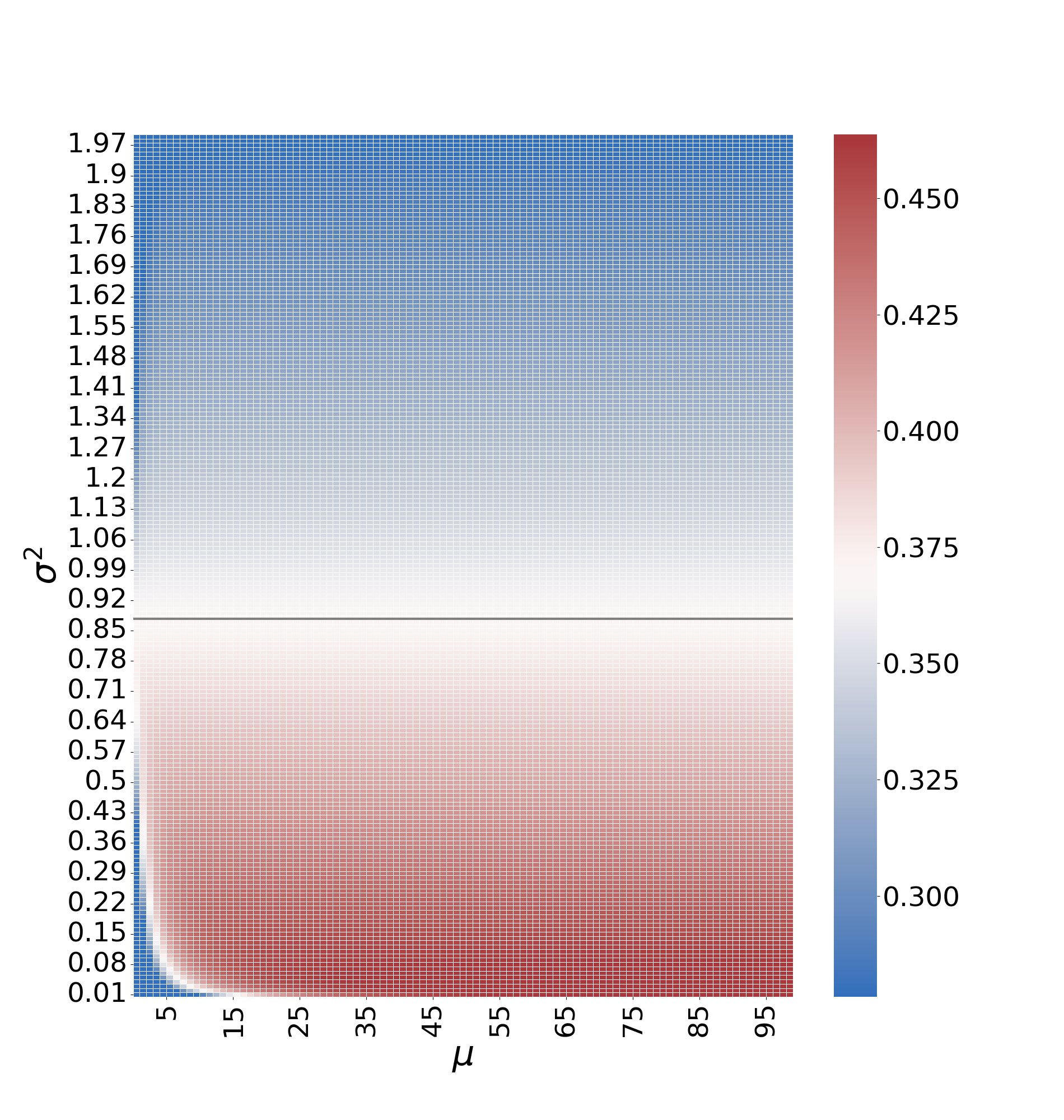}
    \caption{}
    \label{fig:L22}
    \end{subfigure}
    \begin{subfigure}[t]{0.45\columnwidth}
    \centering
    \includegraphics[width=\columnwidth]{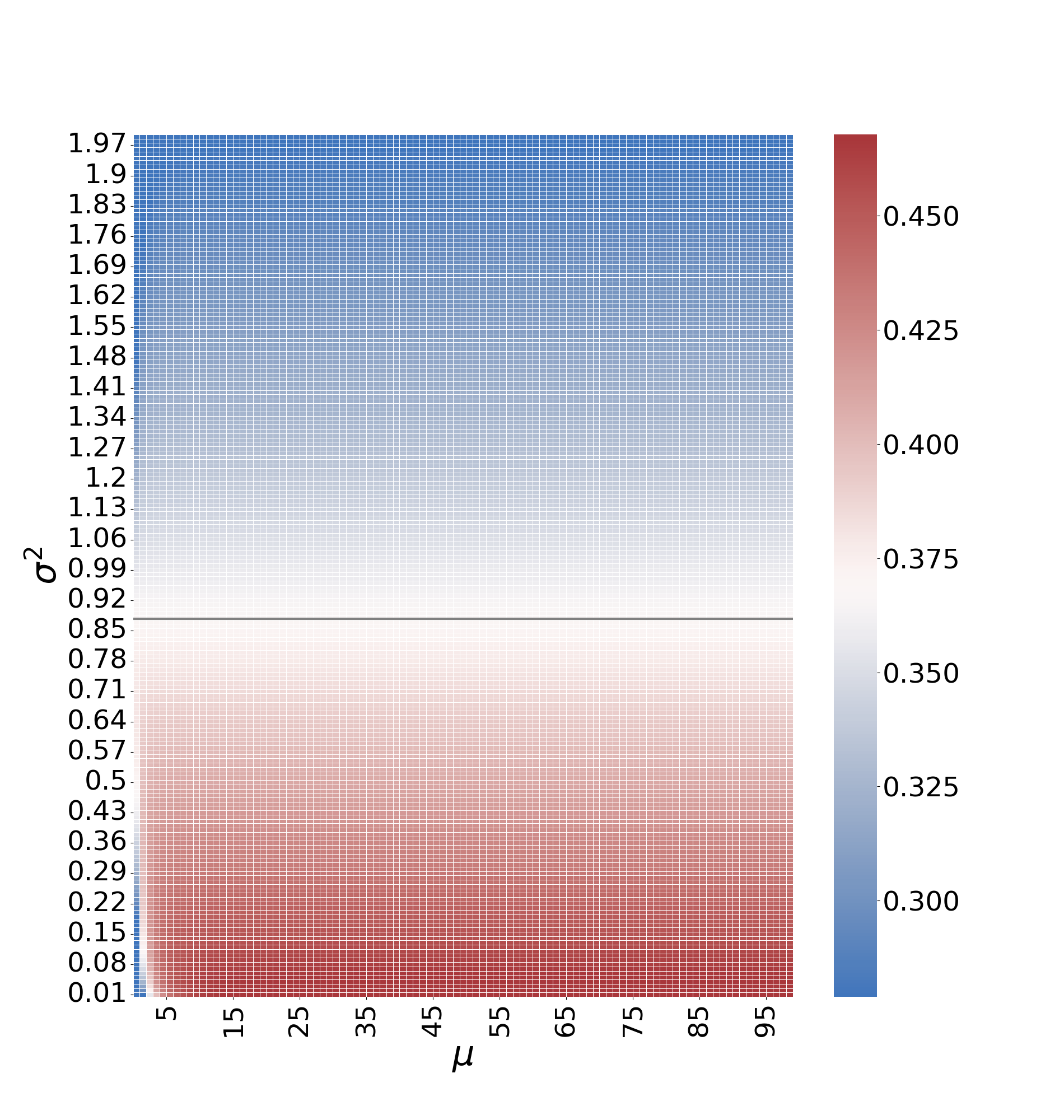}
    \caption{}
    \label{fig:Dkl2}
    \end{subfigure}
    \caption{Heatmaps of
	$C_\cF(\mu, \sigma^2)$
	for 
		(a) $\cF = \fistwo{\cdot}{\bfF}$ 
	and 
		(b) $\cF = \kldiv{\cdot}{\bfF}$.
	The chosen parameter value is $\alpha = 0.45$, and the  variance of the Gaussian quasi-equilibrium is 
	$1/\beta \approx 0.87$.
	This value is indicated by the grey line.
	The corresponding contraction factor  is 
	$\bigl(\frac{1}{2} + \beta\bigr)^{-2} \approx 0.37$
	as computed in \eqref{eqn:mu_n_sigma_n}.
	As $\sigma^2 \to 0$ after $\mu \to \infty$, the contraction factor $C_\cF(\mu, \sigma^2)$ approaches $(1 + \alpha)^{-2} \approx 0.48$, as computed in \eqref{eq:cont-fac}.}
    \label{fig:contraction}
\end{figure}

\section{Peaks of strongly of log-concave densities}
\label{sec:peaks-log-conc}

The following standard result asserts that strongly log-concave distributions concentrate around the minimiser of their potential.
Since we apply the result for general log-concave densities (not necessarily having full support on $\R^d$), we provide a detailed proof.

\begin{lemma} \label{lem:conc-around-argmin}
	Let $\mu = e^{-V}$ be a $\kappa$-log-concave probability density 
	on $\R^d$ for some $\kappa>0$.
	Assume that $V$ is lower semicontinuous, and 
	set
	$\hat{x} := \argmin V$. Then we have 
	\begin{equation}
		\int_{\R^d}  |x-\hat x|^2 \mu(x) \dd x
			\leq \frac{d}{\kappa} \, .
	\end{equation}
\end{lemma}

\begin{proof} Note first that since $V$ is lower semicontinuous and $\kappa$-convex, it indeed admits a minimiser.
	The proof then consists of two steps.

	\smallskip
	\textit{Step 1.} Assume that 
		$V\colon \R^d\to \R$ is of class $C^2$ 
	and such that $\nabla V$ is Lipschitz. 
	In this case, $\hess V \succcurlyeq \kappa I_d$ 
	and there exists a solution to the Langevin equation
	\[
		dX_t = -\nabla V (X_t) dt + \sqrt{2} dB_t\,, 
		\qquad 
		X_0 = \hat{x} \, . 
	\]
	Using It\^o's formula, 
	the $\kappa$-convexity of $V$, 
	and the fact that $\nabla V(\hat x) = 0$, we find
	\begin{align*}
		\frac12 \ddt \expe {|X_t-\hat{x}|^2} &  
		=
		- \expe{\nabla V(X_t)\cdot(X_t-\hat x)} + d
		\leq 
		- \kappa \expe{|X_t-\hat x|^2} 
		+ d \, .
	\end{align*}
	Hence, 
	\[
		\expe{| X_t-\hat x|^2} \leq \frac{d}{\kappa}
	\]
	for all $t\geq 0$. 
	As $\expe{\norma*{X_t-\hat x}^2} 
	= W_2(\law(X_t), \delta_{\hat x})^2$,
	the conclusion follows by passing to the limit $t\to \infty$, 
	since $W_2(\law(X_t), \mu) \to 0$; see e.g.,~
	\cite[Thm.~11.2.1]{Ambrosio-Gigli-Savare:2008}.

	\smallskip
	\textit{Step 2.} 
	We remove the additional assumptions on $\mu$. 
	To this end, 
	define $\mu_n :=  \mu * \gamma_{\frac{1}{n}}$,
	set $V_n := - \log \mu_n$,
	and $\hat x_n := \argmin V_n$. 
	Then $\mu_n$ is $\frac{n\kappa}{n+\kappa}$-log-concave
	by
	Lemma \ref{lem:convolution}.
	Using the triangle inequality in $L^2(\mu_n)$ 
	and an application of Step 1 to $\mu_n$
	we find 
	\begin{align*}
		\bigg(\int 
		| x-\hat x |^2
		\mu_n(x)
		\dd x 
		\bigg)^{1/2}
		&\leq 
		\bigg(
		\int 
			| x - \hat x_n |^2
			\mu_n(x)
		\dd x 
		\bigg)^{1/2}
		+ 
			|\hat x_n - \hat x|
		\\
		& \leq 
		\bigg( d\frac{n+\kappa}{n\kappa}
		\bigg)^{1/2} 
		+
			| \hat x_n - \hat x | \, .
	\end{align*}
	Since $\mu_n$ converges weakly to $\mu$, 
	and $x \mapsto |x - \hat x|^2$ is continuous and bounded from below,
	we have
	\begin{align*}
		\bigg(\int 
		| x-\hat x |^2
		\mu(x)
		\dd x 
		\bigg)^{1/2}
		\leq \liminf_{n \to \infty}
		\bigg(\int 
		| x-\hat x |^2
		\mu_n(x)
		\dd x 
		\bigg)^{1/2} \, .
	\end{align*}
	Thus, to obtain the desired result, it remains to show that 
	$\norma{\hat x_n - \hat x} \to 0$.
			 
For this purpose, fix $\epsilon \in (0,1)$.
It remains to show that there exists 
	$\hat n \geq 1$
such that	
$\mu_n$ attains its maximum in 
a ball of radius $\epsilon$ around $\hat x$
whenever $n \ge \hat n$.

Let $\delta > 0$ be a small parameter, only depending on $\epsilon$, that will be specified later.

First we will argue that 
	$\mu_n$ attains a large value near $\hat x$.
For this purpose,
observe that $\dom{V}$ has non-empty interior, 
since $\mu$ is a log-concave density. 
Take $z \in \dom{V}^\circ$. 
Since $V$ is continuous on its domain,
$V$ is bounded on an open ball around $z$.
Therefore, by convexity of $V$, 
we can find 
$y \in 
	B_{\frac{\epsilon}{2}}(\hat x)
	\cap 
	\dom{V}^\circ$ 
and a radius 
$h > 0$ such that 
	$\mu(x) \ge \mu(\hat x)-\delta$ 
for all $x \in B_h(y)$.
Without loss of generality, we choose 
	$h \leq \min \{ \delta, \frac{\epsilon}{2} \}$.
Observe now that there exists a constant 
	$\hat n \geq 1$
depending only on $h$ and the dimension $d$, 
such that
\begin{align}
	\label{eq:h-choice}
	\int_{B_h(0)} \gamma_{\frac{1}{n}}(x) \dd x \geq 1-h
\end{align}
for all $n \geq \hat n$. 
Hence, for $n\ge \hat n$, \eqref{eq:h-choice} yields
		\begin{align}
			\label{eq:lower-bound-Fn-close}
			\mu_n(y)  
			\geq
			\bigl( \mu(\hat x) - \delta \bigr) (1-h) 
			\ge 
			\bigl( \mu(\hat x) - \delta \bigr) (1-\delta) \,.
	\end{align}

Next we will quantify the fact that 
	$\mu_n$ decreases fast if $|x - \hat x|$ increases.
Indeed, since $V$ is $\kappa$-convex
		and $\hat x = \argmin V$, we have
		$
			V(x) \geq V(\hat x) 
				+ \frac{\kappa}{2} |x - \hat x|^2 
		$
		for all $x \in \R^d$, 
		hence
		\begin{align*}
			\mu(x) 
				\leq 
				e^{- \frac{\kappa}{2} |x - \hat x|^2 } \mu(\hat x) \, .
		\end{align*}
		Therefore, if $|x - \hat x| > \epsilon$, another application of \eqref{eq:h-choice} yields, taking into account that $h \leq \frac{\epsilon}{2}$,
		\begin{align}
			\label{eq:upp-bound-Fn-far}
			\mu_n(x) 
				\leq 
			\sup_{|y - x| \leq h} \mu(y)
				+ h \sup_{y \in \R^d} \mu(y)
				\leq 
			\sup_{|y - \hat x| \geq \frac{\epsilon}{2}} \mu(y)
				+ h \mu(\hat x)
				\leq 
				\big( 
					e^{- \frac{\kappa}{8} \epsilon^2 } 
				+ \delta
				\big)
				\mu(\hat x) \, .
		\end{align}
		Choosing $\delta > 0$ small enough (depending on $\epsilon$),
		it follows by 
		combining \eqref{eq:lower-bound-Fn-close} and \eqref{eq:upp-bound-Fn-far} 
		that 
		\begin{align*}
			\mu_n(y) > \sup_{x: | x - \hat x | > \epsilon } \mu_n(x)\, ,
		\end{align*}
		hence $\hat x_n \in \overline B_{\epsilon}(\tilde x)$
		whenever $n \geq \hat n$, which completes the poof.
\end{proof}

\begin{lemma}\label{lem:min-sum-conve-and-con-smooth}
Let $V, U\colon \R^d \to \R\cup\graf*{+\infty}$ be
strictly convex functions such that 
\begin{enumerate}[(i)]
	\item $V$ is lower semicontinuous and $\alpha$-convex 	
		for some $\alpha > 0$;
	\item $U$ belongs to $C^1(\R^d)$, 
		it admits a minimiser, 
		and
		$\nabla U$ is $\beta$-Lipschitz for some $\beta>0$.
\end{enumerate}
Define
$x = \argmin V$, $y = \argmin U$, and 
$z = \argmin (V + U)$. Then:
\begin{align*}
	\frac{1}{\alpha}\norma*{z-y} & 
	\ge 
	\max \Big\{
	\frac{1}{\beta}\norma*{z - x} \; ,
	\frac{1}{\alpha + \beta} 
	\norma*{{y-x}}
	\Big \}.
\end{align*}
\end{lemma}

\begin{proof}
	Note first that since $V + U$ is lower semicontinuous and $\alpha$-convex, it indeed admits a minimiser.
	\smallskip

	\textit{Step 1.} Assume additionally that $V \in C^1(\R^d)$. 
	Then
	one of the two desired inequalities follows from
	\[
	\alpha \norma{z-x} 
	\le 
	\norma*{\nabla V(z) } 
	= 
	\norma*{\nabla U(z) } 
	\le 
	\beta \norma*{z-y} \, . 
	\]
	The other one follows by combining this inequality with the triangle inequality
	$\norma{ z - x} \geq \norma{ y - x} - \norma{ z - y}$.

	\smallskip
	\textit{Step 2.} 
	We now remove the additional assumption that $V \in C^1(\R^d)$. 
	For $\lambda > 0$, 
	we consider the Moreau-Yosida approximation $V_\lambda$ of $V$ defined by
	\[
		V_\lambda(x) = \inf_{y\in \R^d} \graf*{V(y) + \frac{1}{2\lambda} \norma{x-y}^2}\, .
	\]
	It is classical that $V_\lambda$ is of class $C^1$ and $\alpha_\lambda$-convex with $\alpha_\lambda \downarrow \alpha > 0$ as $\lambda \to 0$ (cf.~\cite[Prop. 3.1]{cle-2009}).
	Clearly, $x = \argmin V_\lambda$. 
	
	Write 
		$z_\lambda := \argmin \tond*{V_\lambda + U}$. 
	An application of Step 1 yields
	\[
			\frac{1}{\alpha_\lambda}\norma*{z_\lambda-y} 
		\ge 
		\max \Big\{
		\frac{1}{\beta}\norma*{z_\lambda - x} \; ,
		\frac{1}{\alpha_\lambda + \beta} 
		\norma*{{y-x}}
		\Big \} \, .
	\]
	Therefore, to derive the desired conclusion, it remains to prove that $z_\lambda \to z$ as $\lambda \to 0$.
	
	\smallskip
	To show this, 
	define
	$\tilde z_\lambda 
		:= \argmin_y\big\{V(y) + \frac{1}{2\lambda}\norma*{y-z_\lambda}^2\big\}
	$, so that
	\begin{align}\label{eq:Vl}
		V_\lambda(z_\lambda) = V(\tilde z_\lambda) + \frac{1}{2\lambda}\norma{z_\lambda - \tilde z_\lambda}^2 \, .
	\end{align}
	We claim that 
	there exists a compact set $\cC$ such that 
	$z_\lambda, \tilde z_\lambda \in \cC$
	for all $\lambda \in (0, 1]$.
	Let us show this.
	Since
	$V_\lambda \leq V$,
	$z_\lambda = \argmin (V_\lambda + U)$,
	and 
	\eqref{eq:Vl},
	we obtain
	\begin{align}
		V(z) + U(z) 
		& \ge 
		  V_\lambda(z)+U(z)
		\ge  
		V_\lambda(z_\lambda)+U(z_\lambda)
		\geq 
		V(\tilde z_\lambda) + \frac{1} {2\lambda}\norma*{z_\lambda - \tilde z_\lambda}^2 + U(z_\lambda)\, .
		\label{eq:lower-bound-z-zlambd-ztillambd}
	\end{align}
	Using this inequality
	and
	the fact that $x = \argmin V$ and $y = \argmin U$,
	we find 
	\begin{align*}
		V(z) + U(z) 
		\ge  
		V(x) 
		+ \frac{1} {2\lambda}
			\norma*{z_\lambda - \tilde z_\lambda}^2 
		+ U(y)\, .
	\end{align*}
	Consequently,
	\begin{equation}\label{eq:bound-zlambd-ztillambd}
		\norma*{z_\lambda - \tilde z_\lambda}^2
			\le 2\lambda M
			\, ,
				\quad 
			\text{ where }
			M :=  V(z)+U(z)-V(x)-U(y)
			\, 			.
	\end{equation}
	Since $x = \argmin V$ and $V$ is $\alpha$-convex,
	$y = \argmin U$, and \eqref{eq:lower-bound-z-zlambd-ztillambd}, 
	we deduce 
	\begin{align*}
		\frac12 | \tilde z_\lambda  - z |^2
		& \leq 
		| \tilde z_\lambda  - x |^2
			+
		| x - z |^2
		\leq 
		\frac{2}{\alpha} V(\tilde z_\lambda)
		+ 
		| x - z |^2
		\\& \leq
		\frac{2}{\alpha} 
		\Big( 
			V(\tilde z_\lambda)
			+ 
			\frac{1} {2\lambda}
				\norma*{z_\lambda - \tilde z_\lambda}^2
			+ U(z_\lambda) - U(y)
		\Big)
		+ 
		| x - z |^2
		\\& \leq
		\frac{2}{\alpha} 
		\Big( 
		V(z) + U(z) - U(y)
		\Big)
		+ 
		| x - z |^2 \, .
	\end{align*}
	Together with \eqref{eq:bound-zlambd-ztillambd}, this estimate yields the claim.	

	\smallskip
	Fix $\epsilon > 0$.
	Since $U$ is uniformly continuous on $\cC$,
	there exists $\delta \in (0,\frac{\epsilon}{2})$ such that 
	\begin{align}
		\label{eq:unif-cont}
		\abs*{U(x_1)-U(x_2)} \le 
		\frac{\alpha \epsilon^2}{8}
	\end{align}
	for all
		$x_1,x_2 \in \cC$ with $\norma*{x_1-x_2}\le \delta$.
	Define
	$
		\hat \lambda 
			:= \min\bigl\{1, 
				\frac{\delta^2}
					{2M} \bigr\}
	$.
	To complete the proof, we shall show that $\norma{z-z_\lambda} \le \epsilon$ whenever $\lambda \le \hat \lambda$.

	Note first that 
		$\norma*{z_\lambda-\tilde z_\lambda}\le \delta$
	for all $0 < \lambda \leq \hat \lambda$ 
	by
	\eqref{eq:bound-zlambd-ztillambd} and the definition of $\hat \lambda$.
	Using \eqref{eq:lower-bound-z-zlambd-ztillambd},
	\eqref{eq:unif-cont},
	the $\alpha$-convexity of $V+U$ 
	and the fact that
	 $z = \argmin (V+U)$,
	we further deduce that 
	\begin{align*}
		V(z) + U(z) 
		\ge V(\tilde z_\lambda) +U(z_\lambda) 
		&
		 \ge V(\tilde z_\lambda) +U(\tilde z_\lambda)
			- \frac{\alpha \epsilon^2}{8}
		\\
		& 
		\ge V(z)+ U(z) +\frac{\alpha}{2} \norma*{z-\tilde z_\lambda}^2 
			- \frac{\alpha \epsilon^2}{8} \, .
	\end{align*}
	This implies $\norma{z-\tilde z_\lambda}\le \frac{\epsilon}{2}$. 
	Since 
		$\norma*{z_\lambda - \tilde z_\lambda}
			\le \delta 
			< \frac{\epsilon}{2}$,
	we obtain the desired result.
	\end{proof}

\begin{appendix}
	\section{\texorpdfstring{$L^p$}{Lp}-transport information inequalities}\label{app-p-transp-inf} 
	
	Let $\mu,\nu = e^{-H} \mu \in L^1_+(\R^d)$  be probability densities. 
	In this section, we assume for simplicity that 
	$H\in C^1(\R^d)$. 
	For $p \in [1, \infty)$ 
	we consider the $L^p$-relative Fisher information $\fisp{\nu}{\mu}$ defined by
	\begin{equation}
		\fisp{\nu}{\mu} 
		= 
		\int_{\R^d} 
			\norma*{\nabla H}^p 
		\dd \nu 
		= 
		\int_{\R^d} 
			\norma*{\nabla \log 
		\Bigl(\frac{\ddd\nu}{\ddd\mu}\Bigr)}^p 
		\dd \nu \,.
	\end{equation}
	Note that $\fistwo{\nu}{\mu}$ 
	is the classical relative Fisher information, 
	while the $L^\infty$-relative Fisher information can be recovered in the limit:
		$\bigl(
				\fisp{\nu}{\mu}
		\bigr)^{1/p} \to \fisinf{\nu}{\mu}$
	as $p \to \infty$.
	
	The following result is the $L^p$-version of Theorem \ref{thm:funct-ineq}.
	
	\begin{theorem}
		\label{thm:p-transp-inf}
		Let $\mu \in L_+^1(\R^d)$ be a $\kappa$-log-concave probability density for some $\kappa > 0$.
		Then the $p$-\emph{transport-information inequality}
		\begin{equation}
			\label{eq:p-transp-inf}
			W_{p} (\mu,\nu) \le \frac{1}{\kappa} \, \tond*{\fisp{\nu}{\mu}}^\frac{1}{p} 
		\end{equation}
		holds for all probability densities 
		$\nu = e^{-H} \mu \in L_+^1(\R^d)$
		with $H\in C^1(\R^d)$.
	\end{theorem}
	
	\begin{proof}
	We assume that $\fisp{\nu}{\mu} <\infty$, since otherwise there is nothing to prove. 
	Moreover, we assume that $p>1$, noting that the case $p=1$ follows from by passing to the limit $p \to 1$.
	The proof is an adaptation of the proof of Theorem \ref{thm:main-abstract-coupl} with an additional approximation argument.
		
	\smallskip
	
	\textit{Step 1.} \
	Suppose first that $\mu = e^{-U}$ for some 
	$U \in C^2(\R^d)$   
	such that $\nabla U$ is Lipschitz,
	and that $H\in C^1(\R^d)$ is also Lipschitz.
	As in the proof of Theorem \ref{thm:main-W-infty},
	there exists a unique strong solution to the following system of SDEs, driven by \emph{the same Brownian motion} $B_t$, for all times $t \geq 0$:
	\begin{align}
		\dd X_t 
		& = 
		-\nabla U(X_t) \dd t 
		+
		\sqrt{2} \dd B_t \, , \quad 
		& X_0 & \sim \nu\, ,
		\\
		\dd Y_t 
		&	= 
		-\nabla U(Y_t) \dd t 
		- \nabla H(Y_t) \dd t 
		+
		\sqrt{2} \dd B_t\, , 
		\quad 
		& Y_0 & = X_0\,.
	\end{align}
	Subtracting these equations in their integral form
	and setting
	$Z := X - Y$, we 
	infer that $t \mapsto Z_t$ is differentiable and
	\begin{align*}
		\frac{1}{p} \ddt  |Z_t|^p
		& =  \norma*{Z_t}^{p-2}
			\graf*{- \bip{X_t-Y_t,
				\nabla U (X_t) - \nabla U(Y_t) }
		+ \bip{Z_t, \nabla H(Y_t)}}
		\\
		& \leq
		-\kappa\abs*{Z_t}^p + \norma*{Z_t}^{p-1}\,\norma*{\nabla H(Y_t)} \, .
	\end{align*}
	It follows that $M_T := \sup_{t \in [0,T]}|Z_t|$ can be bounded by a deterministic constant depending only on $\kappa$, $T$, and the Lipschitz constant of $H$.
	Moreover, 
	\begin{align*}
		\frac{1}{p} \ddt \Bigl( e^{\kappa p t} |Z_t|^p \Bigr)
		= 
		e^{\kappa p t} 
		\Bigl(
			\frac{1}{p} \ddt  |Z_t|^p
			+
			\kappa\abs*{Z_t}^p
		\Bigr)
		\leq
		e^{\kappa p t} | Z_t|^{p-1} \,|\nabla H(Y_t)| \, .
	\end{align*} 
	Integrating this inequality yields
	\begin{align*}
		e^{\kappa p T} |Z_T|^p
		\leq
			p \int_0^T 
				e^{\kappa p t} |Z_t|^{p-1}\, |\nabla H(Y_t)| 	
			\dd t \, ,
	\end{align*}
	hence, using Fubini's theorem and H\"older's inequality,
	\begin{align*}
		 \E |Z_T|^p
		\leq
		p \int_0^T e^{-\kappa p (T-t)} 
			\expe{\norma*{Z_t}^{p}}^{\frac{p-1}{p}} 		
			\, 
			\expe{\norma*{\nabla H(Y_t)}^{p}}^{\frac{1}{p}} 		
			\dd t \, .
	\end{align*}
	Since $Y_t \sim \nu$, we have $\fisp{\nu}{\mu} =\expe{\norma*{\nabla H(Y_t)}^{p}}$. Consequently, we obtain
	\begin{align*}
		\E |Z_T|^p
		\leq 
		\frac{1}{\kappa}
		\fisp{\nu}{\mu}^{\frac{1}{p}} 
		\sup_{0 \leq t \leq T} 
		\expe{\norma*{Z_t}^{p}}^{\frac{p-1}{p}} 		
		\, ,
	\end{align*}
	and therefore, for all $t \geq 0$,
	\begin{align*}
		W_p(\law(X_t), \law(Y_t))
		\leq 
		\bigl( \E |Z_t|^p\bigr)^{\frac{1}{p}}
		\leq 
		\frac{1}{\kappa}
		\fisp{\nu}{\mu}^{\frac{1}{p}}  \, .
	\end{align*}
	The conclusion follows by letting $t\to \infty$ and the joint lower semicontinuity of $W_p$ with respect to weak convergence.

	\smallskip
	\textit{Step 2.} 
	We now remove the extra assumptions on $\mu$, as in Step 2 of the proof of Theorem \ref{thm:main-abstract-coupl}.
		To this end, set 
	$\mu_n = \mu * \gamma_{\frac{1}{n}I_d}$ 
	and define the probability density 
	$\nu_n 
	\propto e^{-H}\mu_n$.
	Note that $U_n = - \log \mu_n$ is smooth
	with
	\[
			\kappa_n I_d \preccurlyeq \hess U_n \preccurlyeq nI_d			
	\]
	and $\kappa_n \coloneqq \tond*{\frac{1}{\kappa} + \frac{1}{n}}^{-1}$
	by Lemma \ref{lem:convolution} and \ref{lem:upp-bound-hess-pot-gauss-conv}.
	Therefore, we are in a position to apply Step 1 and we obtain the bound 
	$W_p(\mu_n, \nu_n) \le \frac{1}{\kappa_n} \tond*{\fisp{\nu_n}{\mu_n}}^{\frac{1}{p}}$.
	Note that $\mu_n \to \mu$ weakly and by Lemma \ref{lem:strong-weak-conv-log-conc-gauss-conv},
	$\nu_n \to \nu$ weakly too. 
	Hence, since $\nabla H \in C_b(\R^d)$ by assumption, it follows that
	\[
			\fisp{\nu_n}{\mu_n} = \int_{\R^d} \norma*{\nabla H}^p \dd \nu_n \to \int_{\R^d} \norma*{\nabla H}^p  \dd\nu = \fisp{\nu}{\mu}\, .
	\]
	The desired conclusion follows by letting $n\to \infty$ and by the joint lower semicontinuity of $W_p$ with respect to weak convergence.

	\smallskip

	\textit{Step 3.} 
	In this step we remove the additional requirement that $H \in C^1(\R^d)$ is Lipschitz, but we assume instead that it is bounded.
	For $R>0$, let $\phi_R\colon \R^d\to [0,1]$ be a smooth function such that
	\[
	\phi_R (x)  = 
			\begin{cases}
			 1 & \text{ if } \norma{x} \leq R
			\\
			0 & \text{ if }  \norma{x} \ge R +1
		\end{cases}
	\]
	and $\norma*{\nabla \phi_R} $ is uniformly bounded by a constant independent of $R$.
	We consider the functions $H_R \coloneqq \phi_R H$ and the probability measures $\nu_R \propto e^{-H_R} \mu$. 
	Note that
	\[
	{\nabla H_R} = \phi_R \nabla H + \nabla \phi_R H\, , 
	\]
	which implies that $H_R$ is Lipschitz. 
	Hence, by the previous step, 
	\begin{equation}\label{eq:p-trans-in-step-boundH}
		W_p\tond{\mu, \nu_R} \leq \frac{1}{\kappa}\tond*{\fisp{\nu_R}{\mu}}^{\frac{1}{p}} \, .
	\end{equation}
	Notice also that for all $f\in C_b(\R^d)$ we have by the dominated convergence theorem that 
	\[
		\int f e^{-H_R} \dd\mu \to \int f e^{-H} \dd\mu  = \int f \dd\nu
	\]
	as $R\to \infty$, which implies that $\nu_R\to \nu$ weakly.
	 Moreover, 
	 using that $|\nabla H_R| \leq C(1 + |\nabla H|)$ for some constant $C < \infty$ not depending on $R$,
	 another application of the dominated convergence theorem yields
	\[
		\int \norma*{\nabla H_R}^p e^{-H_R} \dd\mu 
		\to \int \norma*{\nabla H}^p e^{-H} \dd\mu 
		= \fisp{\nu}{\mu}\, .  
	\]
	 The desired conclusion follows by passing to the limit $R\to \infty$ in \eqref{eq:p-trans-in-step-boundH}.

	\smallskip
	\textit{Step 4.} Finally, we remove the assumption that $H$ is bounded. 
	To this end, let $j\colon\R \to [0\,\infty)$ be a smooth symmetric mollifier supported in $[-1,1]$. 
	For an integer $n\ge 2$, consider the function 
		$\phi_n(x) \coloneqq \graf*{\quadr*{(\cdot)\wedge n\vee (-n)} *j}(x)$.
	Note that $\phi_n$ is smooth, non-decreasing, $1$-Lipschitz and such that $\abs{\phi_n(x)} \le \abs{x}$ and
       \begin{equation}
        \phi_n(x) = \begin{cases}
             -n &\text{ if } x \leq -(n+1),
            \\
            x &\text{ if } |x| \leq  n-1,
            \\
            n &\text{ if } x \geq n+1.
        \end{cases}
    \end{equation}
    Then, define the function 
    $H_n \coloneqq \phi_n \circ H$ and the probability density $\nu_n \propto e^{-H_n} \mu$. Note that $H_n$ converges pointwise to $H$ as $n \to \infty$. Moreover, $e^{-H_n} \leq 1 + e^{-H} \in L^1(\mu)$. By dominated convergence, we have that $e^{-H_n}\to e^{-H}$ in $L^1(\mu)$, which implies that $\nu_n$ to $\nu$ weakly.
    Note also that $H_n$ is bounded, and so by the previous step we have that
    \[
        W_p(\nu_n, \mu) \leq \frac{1}{\kappa} \, 
		\bigl(\fisp{\nu_n}{\mu}\bigr)^\frac{1}{p} \, .
    \]
    The desired conclusion follows by letting $n\to \infty$ if we show that 
    \[
        \limsup_{n\to \infty} \int_{\R^d} 
			\norma*{\nabla H_n}^p e^{-H_n} \dd\mu 
			\leq
		\int_{\R^d} \norma*{\nabla H}^p e^{-H} \dd\mu\, .
    \]
    To this end, notice first that $\norma*{\nabla H_n} = \norma*{\phi_n'(H)\nabla H} \leq \norma*{\nabla H}$.
    Hence, we have
    \begin{equation*}
        \begin{split}
            \int_{\R^d} 
				\norma*{\nabla H_n}^p e^{-H_n} 
			\dd \mu 
		& \leq 
			\int_{\R^d} \norma*{\nabla H}^p e^{-H_n}
				\indic_{H^{-1}([-n-1,n+1])} 
			\dd \mu 
        \\
        & \leq 
			\int_{\R^d}
				\norma*{\nabla H}^pe^{-H} 
			\dd \mu 
			+ e^2\int_{\R^d} 
				\norma*{\nabla H}^p e^{-H} 
					\indic_{H^{-1}([n-1,n+1])} 
				\dd \mu \, .
        \end{split}
    \end{equation*}
    The desired conclusion then follows since $ \int_{\R^d} \norma*{\nabla H}^p e^{-H} \indic_{H^{-1}([n-1,n+1])} \dd \mu \to 0$ as $n\to \infty$ by dominated convergence.
	\end{proof}

\end{appendix}

	\subsection*{Acknowledgement} { 
	This research was funded in part by the Austrian Science Fund (FWF)
	project \href{https://doi.org/10.55776/F65}{10.55776/F65} and the 
	Austrian Academy of Science, DOC fellowship nr.~26293. 
	}

	\small

	\bibliographystyle{my_alpha}
	\bibliography{Fisher}
	
\end{document}